\documentclass[11pt]{amsart}
\usepackage{amssymb}
\usepackage{hyperref}
\usepackage{cleveref}
\usepackage[colorinlistoftodos,bordercolor=orange,backgroundcolor=orange!20,linecolor=orange,textsize=scriptsize]{todonotes}

\usepackage{graphicx, enumerate, url}
\usepackage{amssymb,mathtools}

\usepackage{algorithm}
\usepackage{algpseudocode}
\usepackage{color}
\usepackage{tikz}
\usetikzlibrary{3d,calc}

\numberwithin{equation}{section}
\numberwithin{algorithm}{section}

\theoremstyle{plain}
\newtheorem{theorem}{Theorem}[section]
\newtheorem{proposition}[theorem]{Proposition}
\newtheorem{lemma}[theorem]{Lemma}
\newtheorem{corollary}[theorem]{Corollary}

\theoremstyle{definition}
\newtheorem{definition}[theorem]{Definition}

\theoremstyle{remark}

\let\H\undefined

\DeclareMathOperator{\tr}{Tr}
\DeclareMathOperator{\diag}{diag}

\DeclareMathOperator{\H}{H}

\DeclareMathOperator{\adj}{adj\,}

\newcommand{\N}{\mathbb{N}}
\newcommand{\C}{\mathbb{C}}

\newcommand{\R}{\mathbb{R}}

\newcommand{\x}{\mathbf{x}}
\newcommand{\X}{\mathbf{X}}
\newcommand{\y}{\mathbf{y}}
\newcommand{\z}{\mathbf{z}}

\newcommand{\cC}{\mathcal{C}}

\newcommand{\cH}{\mathcal{H}}

\newcommand{\cU}{\mathcal{U}}
\newcommand{\cV}{\mathcal{V}}

\newcommand{\cZ}{\mathcal{Z}}
\newcommand{\ba}{\mathbf{a}}
\newcommand{\bb}{\mathbf{b}}
\newcommand{\bc}{\mathbf{c}}
\newcommand{\bd}{\mathbf{d}}

\newcommand{\bi}{\mathbf{i}}
\newcommand{\bn}{\mathbf{n}}

\newcommand{\bp}{\mathbf{p}}

\newcommand{\bs}{\mathbf{s}}
\newcommand{\bt}{\mathbf{t}}
\newcommand{\bu}{\mathbf{u}}
\newcommand{\bv}{\mathbf{v}}
\newcommand{\bw}{\mathbf{w}}
\newcommand{\bx}{\mathbf{x}}

\newcommand{\rB}{\mathrm{B}}
\newcommand{\rC}{\mathrm{C}}
\newcommand{\rD}{\mathrm{D}}
\newcommand{\rE}{\mathrm{E}}

\newcommand{\rH}{\mathrm{H}}

\newcommand{\rL}{\mathrm{L}}

\newcommand{\rS}{\mathrm{S}}

\newcommand{\rU}{\mathrm{U}}
\newcommand{\rV}{\mathrm{V}}

\newcommand{\0}{\mathbf{0}}
\newcommand{\1}{\mathbf{1}}

\makeatletter
\@namedef{subjclassname@2020}{\textup{2020} Mathematics Subject Classification}
\makeatother

\begin{document}
\title[Barrier relaxations of optimal transport problems]
{Barrier relaxations of the classical and quantum optimal transport problems}
\author{Shmuel Friedland}
\address{
 Department of Mathematics, Statistics and Computer Science,
 University of Illinois at Chicago, Chicago, Illinois 60607-7045,
 USA, \texttt{friedlan@uic.edu}
 }
 \date{November 10, 2025}
\subjclass[2010]
  	{}
  	

\begin{abstract}  In the last fifteen years a significant progress was achieved by considering an entropic relaxation of the classical multi-partite optimal transport problem  (MPOTP).  The entropic relaxation gives rise to the rescaling problem of a given tensor.   This rescaling can be achieved fast with  the Sinkhorn type algorithms.  Recently, it was shown that a similar approach works for the quantum MPOTP.  However, the analog of the rescaling Sinkhorn algorithm is much more complicated than in the classical MPOTP. 
In this paper we show that the interior point method (IPM) for the primary and dual problems of classical and quantum MPOTP problems can be considered as barrier relaxations of the optimal transport problems (OTP).   It is well known that the dual of the OTP are advantageous as it has much less variables than the primary problem.  The IPM for the dual problem of the classical MPOTP are not as fast as the Sinkhorn type algorithm.   However,  IPM method for the dual of the quantum MPOTP seems to work quite efficiently. 
\end{abstract}


\maketitle
\tableofcontents
\section{Introduction}\label{sec:intro}
\subsection{Discrete classical optimal transport problem}\label{subsec:cdotp}
The  discrete classical bi-partite optimal transport problem (BPOTP) was introduced by Hitchcock and Kantorovich  \cite{Hit41,Kan42}.   It is the  following linear programming problem (LP).   Let $\Delta^n\subset\R^n$ the simplex of column probablity vectors, and $[n]=\{1,\ldots,n\}$.  For $\x=(x_1,\ldots,x_n)^\top\in\Delta^n$ denote by $E_S(\x)=-\sum_{i=1}^n x_i\log x_i$ the Shannon entropy of $\x$.
Assume that $\bp_i\in\Delta^{n_i}$ for $i=1,2$.  
Let $P=(\bp_1,\bp_2)$ and  $\rV(P)$ be the convex set of all (coupling) probability matrices with marginals $\bp_1,\bp_2$:
\begin{equation}\label{defUPmat}
\rV(P)=\{V=[v_{j_1,j_2}]\in\R_+^{n_1\times n_2}, \sum_{j_2=1}^{n_2} v_{j_1,j_2}=p_{j_1,1},  \sum_{j_1=1}^{n_1} v_{j_1,j_2}=p_{j_2,2}\}.
\end{equation}
Here $\R_+^n\supset \R_{++}^n$ stands for the cone of nonnegative vectors and its interior of positive vectors in $\R^n$.
Let $C=[c_{j_1,j_2}]\in \R^{n_1\times n_2}$ be the cost matrix of transporting a unit $j_1\in[n_1]$ to $j_2\in[n_2]$.  The BPOTP is
\begin{equation}\label{TOTmat}
\tau(C,P)=\min_{V\in\rV(P)}\tr C^\top V. 
\end{equation}
This problem can be stated  in terms of maximum flow \cite{PW09}.  Hence,  for $n_1=n_2=n$ the complexity of finding $\tau(C,P)$ is $O(n^3 \log n)$.  It turns out that  $\tau(C,P)$ has many recent applications in machine learning
\cite{AWR17,ACB17,LG15,MJ15,SL11}, statistics  \cite{BGKL17,PZ16,SR04} and computer vision \cite{BPPH11,RTG00}.   For $n>10^3$ the cost of computing $\tau(C,P)$ becomes prohibitively high if one uses the maximal flow algorithm.  

In a breakthrough paper Cuturi \cite{Cut13} suggested an entropic relaxation of the BPOTP:
\begin{equation}\label{EMOT}
\begin{aligned}
\tau(C,P,\varepsilon)= \min_{V\in \rV(P)} \tr C^\top V  -\varepsilon\rE_S(V).
\end{aligned}
\end{equation}
Recall that 
\begin{equation}\label{bctauin}
\tau(C,P,\varepsilon)\le \tau(C,P)\le \tau(C,P,\varepsilon)+\varepsilon\log(n_1n_2).
\end{equation}
(See part (a) of Lemma  \ref{minuecbpotplem}.)
As $-\rE_S$ is strictly convex, the minimal solution to \eqref{EMOT} is unique.   For a matrix $A=[a_{j_1,j_2}]\in\R^{n_1\times n_2}$ denote by $\exp(\circ A)\in \R_{++}^{n_1\times n_2}$ the entry-wise exponential matrix $[\exp(a_{j_1,j_2})]$.  Then the minimal solution to the entropic relaxation of OTP is given by 
\begin{equation}\label{minsolebtot}
\begin{aligned}
&U=\exp\left(\circ(\hat U_1+\hat U_2-\frac{1}{\varepsilon} C)\right)\in \rV(P),\\
&\hat U_1=\bu_1\1_{n_2}^\top,  \hat U_2=\1_{n_1}\bu_2^\top, \bu_1\in\R^{n_1}, \bu_2\in\R^{n_2}, \1_{n_1}^\top\bu_1-\1_{n_2}^\top\bu_2=0,\\
&\1_n=(1,\ldots,1)^\top \in\R^n. 
\end{aligned}
\end{equation}
That is, the minimal solution is a rescaling of $\exp\left(\circ(-\frac{1}{\varepsilon}C)\right)$.  This solution can be obtained by Sinkhorn's algorithm \cite{Sin64} by alternatively rescaling the rows and columns of $U_k=\exp\left(\circ(\hat U_{1,k}+\hat U_{2,k}-\frac{1}{\varepsilon} C)\right)$.  It turns out that fixing the error $\delta$ and $n_1=n_2=n$,  the complexity of modified Sinkhorn's  algorithm is $O(n^2)$ ignoring logarithmic factors \cite{AWR17}.   In this paper by arithmetic complexity of an algorithm we mean an estimate of number of  flops to carry out this algorithm,  see Appendix \ref{sec:flop}.

Similar results were obtained for entropic relaxation of multi-partite optimal transport problems that will be discussed later \cite{BCN19,TDGU20,HRCK,LHCJ22, Fri20}.
\subsection{Finite dimensional quantum optimal transport}\label{subsec:fdqotp} 
Let $\cH$ be an $n$-dimensional vector space with the inner product $\langle \x,\y\rangle$.
Let $\rB(\cH)\supset \rS(\cH)\supset \rS_{+}(\cH)\supset \rS_{++}(\cH)$ be the space of linear operators, the real subspace of self-adjoint operators, the cone of positive semidefinite operators,  and the open set of positive definite operators respectively.
By choosing a fixed orthonormal basis $|0\rangle,\ldots,|n-1\rangle$ we identify $\cH$  with $\C^n$, where $\langle\x,\y\rangle=\x^*\y$.  Then the above linear operators on $\cH$ correspond to $\C^{n\times n}\supset \rH_n\supset \rH_{n,+}\supset\rH_{n,++}$: the algebra of $n\times n$ complex valued matrices, the Hermitian matrices, the positive semidefinite matrices, and the positive definite matrices respectively.  Denote by $\rH_{n,+,1}\supset \rH_{n,++,1}$ the convex set of density matrices (having trace one) and its interior respectively.   For $A\in \rH_n$ denote by 
$$\lambda_{\max}=\lambda_1(A)\ge \cdots\ge \lambda_n(A)=\lambda_{\min}(A),  \lambda(A)=(\lambda_1(A),\ldots,\lambda_n(A))^\top.$$ 
the $n$ eigenvalues of $A$. Recall that the spectral norm of $A$ is $\|A\|_2=\max(\lambda_{\max}(A),-\lambda_{\min}(A))$, and the Frobenius norm $\|A\|_F=\sqrt{\sum_{j=1}^n \lambda_j^2(A)}$.

Assume that $\cH_i$ is a Hilbert space of dimension $n_i$ for $i\in[d]$.  Let $\cH=\otimes_{i=1}^d \cH_i$.  
Then $\C^{\bn}=\otimes_{i=1}^d\C^{n_i}$ is a representation of $\cH$,  and $\rB(\cH)\supset \rS(\cH)$ are represented by $\C^{\bn\times \bn}\supset   \rH_{\bn}$ respectively.  (Here we identify $\H_{n_1\cdots n_d}$ with $2d$-mode Hermitian tensors $\rH_{\bn}$.)

Let $\rho_i\in\rH_{n_i,+,1}$ for $i\in[d]$, for $d>1$.   The coupling set of $R=(\rho_1,\ldots,\rho_d)$ is the closed convex set of all $d$-partite states whose partial traces are $\rho_i$ for $i\in[d]$  respectively:
\begin{equation}\label{couprs}
\Gamma(R)=\{\rho\in\rH_{\bn,+}, \tr_{\hat i} \rho=\rho_i, i\in [d] \},  R=(\rho_1,\ldots,\rho_d)\}.
\end{equation}
Here, $\tr_{\hat i}\rho$ is the sum of all partial traces of $\rho$ except on the index $i$.
(Observe that if $\rho\in \Gamma(R)$ then $\tr\rho=1$.)
The quantum multi-partite optimal transport optimal transport problem for a given $C\in \rH_{\bn}$ is \cite{CGP20,CM20, FECZ22, CEFZ23}:
\begin{equation}\label{qotp}
\kappa(C,R)=\min_{\rho\in\Gamma(R)} \tr C\rho.
\end{equation}
It is a semidefinite programming problem  (SDP) \cite{VB96} that can be solved efficiently using the interior point programming (IPM) \cite{CEFZ23}.  For $\rho\in\rH_{n,+,1}$ denote by $E_N(\rho)=-\sum_{i=1}^n \lambda_i(\rho)\log\lambda_i(\rho)$ the von Neumann entropy of $\rho$.  The entropic relaxation of the quantum $d$-partite problem is 
\begin{equation}\label{eqotp}
\kappa(C,R,\varepsilon)=\min_{\rho\in\Gamma(R)} \tr C\rho -\varepsilon \rE_N(\rho).
\end{equation}
The analog of the inequality \eqref{bctauin} is
\begin{equation}\label{bqkapin}
\kappa(C,R,\varepsilon)\le \kappa(C,R)\le \kappa(C,R,\varepsilon)+\varepsilon\log(n_1n_2).
\end{equation}
It is shown in \cite[equation (1.2)]{FGP23} that the unique solution of \eqref{eqotp} is of the form
\begin{equation}\label{usolqot}
\begin{aligned}
&\rho=\exp\left(\sum_{i=1}^d \hat U_i -\frac{1}{\varepsilon}C\right) \in \Gamma(R), \\
&\hat U_i=(\otimes_{j=1}^{i-1}I_{n_j})\otimes U_i\otimes (\otimes_{k=i+1}^d I_{n_k}),
\quad U_i\in \rH_{n_i}, i\in [d].
\end{aligned}
\end{equation}
The choice of $U_1,\ldots,U_d$ is not unique, and some normalizations on $U_1,\ldots,U_d$ can be imposed \cite{FGP23}.   The step $k$ of the Sinkhorn algorithm consists of choosing a mode $j_k\in[d]$ and updating $U_{j_k}$ such that the corresponding $\rho_{k+1}$ in the above equation satisfies $\tr_{\widehat {j_k}}\rho_{k+1}=\tr \rho_{j_k}$.  While this condition for the classical Sinkhorn algorithm is very simple, in the quantum case it is computationally complex \cite{CGMP24, RR25}.
\subsection{Our contribution}\label{subsec:ourcon}
In this paper we use a relaxation to the classical and quantum MPOTP and their duals using a barrier function, which is the corner stone of the interior point method (IPM).  For the minimum problems the barrier function has to be a smooth convex functions that blow up on the boundary of a given convex domain, and satisfy a few other technical conditions \cite{NN94, Ren01}.   Observe that the Shannon and von Neumann entropies are not barriers as they don't blow up at the boundary.  In \cite{Fri23} we showed how to apply the IPM method for the classical MPOTP.   We briefly explain our approach for the BPOTP.    

The barrier functions on $\R_{++}^n$ and $\R_{++}^{n_1\times n_2}$ are
\begin{equation}\label{bVP}
\begin{aligned}
&\beta(\x)=-\sum_{i=1}^n \log x_i, \quad \x\in R_{++}^n,\\
&\beta(V)=-\sum_{j_1=j_2=1}^{n_1,n_2}\log v_{j_1,j_2},  \quad V=[v_{j_1,j_2}]\in\R_{++}^{n_1\times n_2}.
\end{aligned}
\end{equation}
In what follows we  denote
\begin{equation}\label{defRbal}
(R^{n_1}\times\R^{n_2})_{++,b}=\{(\x,\y)\in\R_{++}^{n_1}\times\R^{n_2}_{++}, \1_{n_1}^\top \x=\1_{n_2}^\top\y\}.
\end{equation}
For $(\bp_1,\bp_2)\in (R^{n_1}\times\R^{n_2})_{++,b}$ we define $V(P)$ as in \eqref{defUPmat}.   We call $\tau(C,P)$ given by \eqref{TOTmat} the (generalized) classical BPOTP.
Let $\rV_o(P)$ be the interior of $\rV(P)$, which consists of all positive  matrices in $\rV(P)$.   The relaxed BPOTP is
\begin{equation}\label{betBOTP}
\tau_{\beta}(C,P,\varepsilon)=\min_{V\in \rV_o(P)} \tr C^T V+\varepsilon \beta(V).
\end{equation}
We will show that this problem has a unique solution of the form
\begin{equation}\label{bBPOTus}
\begin{aligned}
&U=[u_{j_1,j_2}]\in \rV_o(P),  \\
&u_{j_1,j_2}=\frac{\varepsilon}{c_{j_1,j_2}-u_{j_1,1}-u_{j_2,2}}, \sum_{j_1=1}^{n_1}u_{j_1,1}-\sum_{j_2=1}^{n_2}u_{j_2,2}=0.
\end{aligned}
\end{equation}

The following theorem summarizes one of our main results on the barrier relaxation of the classical BPOTP:
\begin{theorem}\label{dbcBPOTP} Let $(\bp_1,\bp_2)\in (R^{n_1}\times\R^{n_2})_{++,b}$, and define
\begin{equation}\label{DmfbBOTP}
\begin{aligned}
&C(\x,\y)=C-\x\1_{n_2}^\top-\1_{n_1}\y^\top,(\x,\y)\in\R^{n_1}\times\R^{n_2},\\
&\hat \rD=\{(\x,\y)\in\R^{n_1\times n_2}, C(\x,\y)>0\},\\
&\rD=\{(\x,\y)\in\hat\rD,\1_{n_1}^\top\x-\1_{n_2}^\top \y=0\}.
\end{aligned}
\end{equation}
The barrier function for $\hat \rD$ is
\begin{equation}\label{barcBOT}
\beta_C(\x,\y)=-\sum_{j_1=j_2=1}^{n_1,n_2} \log(c_{j_1,j_2}-x_{j_1}-y_{j_2}).
\end{equation}
The function 
\begin{equation}\label{mfbBOTP}
\begin{aligned}
&\varphi(\x,\y,C)=\bp_1^\top\x+\bp_2^\top\y-\varepsilon\beta_C(\x,\y) +n_1n_2\varepsilon(1-\log\varepsilon), 
\end{aligned}
\end{equation}
is concave in $\hat\rD$ and strictly concave in $\rD$.  Furthermore,
\begin{equation}\label{maxfBOTP}
\max_{(\x,\y)\in\hat\rD}\varphi(\x,\y,C)=\max_{(\x,\y)\in \rD}\varphi(\x,\y,C)=\tau_{\beta}(C,P,\varepsilon),
\end{equation}
where $\beta$ and $\tau_{\beta}(C,P,\varepsilon)$ are defined by \eqref{bVP} and \eqref{betBOTP} respectively.
The function $\varphi$ achieves its maximum
exactly on the line 
\begin{equation}\label{defLBP}
\rL=\{(\bu_1,\bu_2)+t(\1_{n_1}, -  \1_{n_2}),t\in\R\},
\end{equation}
which intersects $\rD$ at the unique point $(\bu_1,\bu_2)$, which appears in \eqref{bBPOTus}.
\end{theorem}
That is,  the function $\varphi$ is the dual function of the Lagrangian $\tr C^\top V+\varepsilon \beta(V)$.
We show that the above maximum problem is the IPM relaxation of the dual  LP problem corresponding to $\tau(C,P)$, see  part (b) of Lemma \ref{maxu1u2est}:
\begin{equation}\label{dBOTP}
\begin{aligned}
&\tau(C,P)=\max_{(\x,\y)\in\R^{n_1}\times \R^{n_2}, C-\x\1_{n_2}^\top-\1_{n_1}\y^\top\ge 0}\bp_1^\top\x+\bp_2^\top\y=\\
&\sup_{\x,\y\in\hat\rD} \bp_1^\top\x+\bp_2^\top\y=
\sup_{\x,\y\in\rD} \bp_1^\top\x+\bp_2^\top\y
\end{aligned}
\end{equation}

One can apply the IPM method  to the minimum problem \eqref{EMOT} and minus the maximal problem \eqref{maxfBOTP}.   For $n_1=n_2=n$ the minimum problem  \eqref{EMOT} has $(n-1)^2 $ parametrs, while the maximum problem has $2n-1$ parameters.  Thus, the dual maximum characterization is superior.

The beauty and the strength of the IPM method lies in the fact that the value of the parameter $\varepsilon>0$ is determined by the error of the approximation $\delta>|\tau(C,P)-\tau_{\beta}(C,P,\varepsilon)|$ during 
the central path algorithm (CPA) of the IPM \cite{NN94,Ren01}.  The number of iterations of the CPA for the maximum problem \eqref{maxfBOTP} is of order $O(n)$ ignoring the $\log$ factors.  Each iteration of the CPA consists of one step of the Newton method which needs $O(n^3)$ steps for the case $n_1=n_2=n$. Thus, the implementation of IPM has arithmetic comlexity $O(n^4\log (n^2r/\delta))$, which  is significantly inferior to the Sinkhorn algorithm.

We state a Sinkhorn type algorithm to find the unique $\bu_1,\bu_2$. 
It is not an efficient algorithm as the original Sinkhorn algorithm.  It seems to us that the cost this Sinkhorn algorithm is $\log n$ times the cost ot corresponding  classical Sinkhorn algorithm.   

We now discuss the quantum BPOTP. The dual problem to $\kappa(C,R)$ is given by \cite[Theorem 3.2]{CEFZ23}
\begin{equation}\label{dQBOTP}
\kappa(C,R)=\sup_{U_1\in\rH_{n_1},U_2\in \rH_{n_2}, C- U_1\otimes I_{n_2}-I_{n_1}\otimes U_2\succeq 0} \tr (\rho_1 U_1)+\tr(\rho_2 U_2).
\end{equation}
Recall the following conventions: 
$$ A\succeq 0\iff A\in \rH_{n,+},\quad A\succ 0\iff A\in \rH_{n,++}.$$
It is shown in \cite[Theorem 3.2]{CEFZ23} that if $\rho_1, \rho_2$ are positive definite then the supremum in \eqref {dQBOTP} is achieved.   However,  if at least one of the density matrices is singular the supremum may not be achieved \cite[Appendix B.3]{CEFZ23}.  

The classical barrier function on $\rH_{n,++}$ is \cite{NN94,Ren01}:
\begin{equation}\label{barpdm}
\beta(H)=-\log \det H=-\sum_{i=1}^n \log \lambda_i(H), \quad H\in\rH_{n,++}
\end{equation}
It is well known that $\beta(H)$ is a strictly convex function \cite[Theorem 4.8.3]{Frib}.

For $d>1$ define
\begin{equation}\label{balprodHbn}
\begin{aligned}
&(H_{n_1}\times\cdots\times H_{n_d})_{++,b}=\\
&\{(H_{n_1},\dots,H_{n_d}), H_{n_i}\in\rH_{n_i,++}, i\in[d],\tr H_{n_1}=\cdots=\tr H_{n_d}\}.
\end{aligned}
\end{equation}
 In what follows we assume that $(\rho_1,\ldots,\rho_d)\in(H_{n_1}\times\cdots\times H_{n_d})_{++,b}$.    Let $\Gamma(R)$ be defined by \eqref{couprs}.  Then $\Gamma_o(R)=\Gamma(R)\cap\rH_{\bn,++}$. We call $\kappa(C,R)$, given by \eqref{qotp}, the (generalized) quantum MPOTP. 
 
The barrier $\beta$ relaxation of \eqref{qotp} for $\varepsilon>0$ is given by:
\begin{equation}\label{qbMOTP}
\kappa_\beta(C,R, \varepsilon)=\min_{\rho\in\Gamma_o(R)}\tr C\rho + \varepsilon\beta(\rho)
\end{equation}
As $\beta$ is a strictly convex barrier, the above problem has a unique minimal solution $\rho^\star\in \Gamma_o(R)$.  

The following theorem summarizes one of our main results on the barrier relaxation of the quantum BPOTP:
\begin{theorem}\label{dbqBPOTP} Let $(\rho_1,\rho_2)\in (H_{n_1}\times H_{n_2})_{++,b}$.
For a given $C\in\rH_{(n_1,n_2)}$ define
\begin{equation}\label{defqDhD}
\begin{aligned}
&C(X,Y)=C-X\otimes I_{n_2}-I_{n_1}\otimes Y,  (X,Y)\in\rH_{n_1}\times \rH_{n_2},\\
&\hat\rD=\{(X,Y)\in\rH_{n_1}\times \rH_{n_2}, C(X,Y)\in \rH_{(n_1,n_2),++}\},\\
&\rD=\{(X,Y)\in\hat\rD,\tr X-\tr Y=0, \}.
\end{aligned}
\end{equation} 
The barrier function for $\hat \rD$ is
\begin{equation}\label{barqBOT}
\beta_Q(X,Y)=-\log\det C(X,Y),  \quad (X,Y)\in\hat\rD.
\end{equation}
The function 
\begin{equation}\label{mfbqBOTP}
\begin{aligned}
\varphi(X,Y,C)=\tr \rho_1 X+\tr\rho_2 X-
\varepsilon\beta_Q(X,Y)+n_1n_2\varepsilon(1-\log\varepsilon),
\end{aligned}
\end{equation}
is concave in $\hat\rD$ and strictly concave in $\rD$.  Furthermore,
\begin{equation}\label{maxqfBOTP}
\max_{(X,Y)\in\hat\rD}\varphi(X,Y,C)=\max_{(X,Y)\in \rD}\varphi(X,Y,C)=\kappa_{\beta}(C,R,\varepsilon),
\end{equation}
where $\beta$ and $\kappa_{\beta}(C,R,\varepsilon)$ are given \eqref{barpdm} and \eqref{qbMOTP} respectively.
The function $\varphi$ achieves its maximum
exactly on the line 
\begin{equation}\label{defqLBP}
\rL=\{(U_1,U_2)+t(I_{n_1}, -  I_{n_2}),t\in\R\},
\end{equation}
which intersects $\rD$ at the unique point $(U_1,U_2)$.  Furthermore, 
\begin{equation}\label{rhostarex}
\rho^\star=\varepsilon C(X,Y)^{-1}, (X,Y) \in \rL,
\end{equation}
 is the unique point in $\Gamma_o(R)$ at which $\kappa_{\beta}(C,R,\varepsilon)$ is achieved.
\end{theorem}

We will show how to apply the IPM to \eqref{maxqfBOTP}.   The  arithmetic complexity of CPA for the case $n_1=n_2=n$ is $O(n^7\log n)$.  
\subsection{Brief survey of the contents of the paper}\label{subsec:contents}
Section \ref{sec:cbotp} discusses the entropic and barrier relaxations of the classical BPOTP.
In subsection \ref{subsec:cbotp} we summarize briefly the results of Cuturi \cite{Cut13} for the entropic relaxation of the classical BPOPT.  In subsection \ref{subsec:brcbot} we discuss the barrier relaxation of the classical BPOPT. 
We prove Theorem  \ref{dbcBPOTP}.  In subsection \ref{subsec:ipmcBP} we discuss 
the IPM method for approximating $\tau(C,P)$ within precision $\delta$.  To implement the IPM method for the dual problem \eqref{dBOTP} we need to specify an open Euclidean ball  $\rB_o((\x_0,\y_0),r)$, with radius $r$, which contains the maximal solution of $\varphi(\x,\y,C)$ in the domain $\rD$.   Furthermore, we need to augment the barrier \eqref{barcBOT} to the barrier $\hat\beta$  (\eqref{hbcBT}) by adding the barrier of  $\rB_o((\x_0,\y_0),r)$. Theorem \ref{IPMcBP} shows that the arithmetic complexity of the IPM method to compute $\tau(C,P)$ within $\delta$ approximation is $O(n^4\log\frac{n^2r}{\delta})$.  Subsection \ref{subsec:GSa} discusses a generalized Sinkhorn algorithm for computing $\tau_{\beta}(C,P,\varepsilon)$.   It is slower than the classical Sinkhorn algorithm for the entropic relaxation by a factor $\log n$.

Section  \ref{sec:cmotp} discusses the entropic and barrier relaxations of the classical $d$-partite OTP.  We extend the results of subsections \ref{subsec:cbotp} - \ref{subsec:ipmcBP} to MPOTP.   Theorem \ref{IPMcMP} shows that the arithmetic complexity of the IPM method to compute $\tau(C,P)$ within $\delta$ approximation is $O\big(n^{3d/2}\log\frac{n^dr}{\delta}\big)$ for $d\ge 3$.  

Section \ref{sec:qbotp} discusses quantum bi-partite optimal control and its relaxations.  In subsection \ref{subsec:brqbot} we prove Theorem \ref{dbqBPOTP}. 
In subsection \ref{subsec:ipmqBP} we discuss interior point method for finding a barrier approximation of $\kappa(C,R)$ (\eqref{qotp}).   To implement the IPM method for the dual problem \eqref{maxqfBOTP} we need to specify an open Euclidean ball  $\rB_o((X_0,Y_0),r)$, with radius $r$, which contains the maximal solution of $\varphi(X,Y,C)$ in the domain $\rD$.   Furthermore, we need to augment the barrier \eqref{barqBOT}to the barrier $\hat\beta$  (\eqref{hbqBT}) by adding the barrier of  $\rB_o((\x_0,\y_0),r)$. Theorem \ref{IPMqBP} shows that the arithmetic complexity of the IPM method to compute $\kappa(C,R)$ within $\delta$ approximation is $O(n^7\log\frac{n^2r}{\delta})$.  

Section \ref{sec:qmotp} discusses briefly the extensions of our results in Section \ref{sec:qbotp} to quantum multipartite-partite optimal control.   Theorem \ref{IPMqMP} shows that the arithmetic complexity of the IPM method to compute $\kappa(C,R)$ within $\delta$ approximation is $O(n^{7d/2}\log\frac{n^dr}{\delta})$.  
\section{Classical bi-partite optimal control and its relaxations}\label{sec:cbotp}
In this section we discuss first the  known results on the classical BPOTP and its entropic relaxation.  Second we discuss the barrier relaxations the primary and dual of classical BPOTP.
\subsection{An entropic relaxation of the classical bi-partite optimal transport problem}\label{subsec:cbotp}
For $i\in\{1,2\}$ let $X_i$ be a random variables on $\Omega_i$ which has a finite number of values: $X_i:\Omega_i\to [n_i]$.
 Assume that $\bp_i=(p_{1,i},\ldots,p_{n_i,i})^\top$ is the column probability vector that gives the distribution of $X_i$: $\mathbb{P}(X_i=j)=p_{j,i}$ for $i\in[2]$.
 Let $Z$ be a random variable $Z: \Omega_1\times \Omega_2\to [n_1]\times[n_2]$ with contingency matrix (table)  $U\in\R_+^{n_1\times n_2}$ that gives the distribution of $Z$:
$\mathbb{P}(Z=(j_1,j_2))=u_{j_1,j_2}$.  
The classical BPOTP is given by \eqref{TOTmat}.

For $n_1=n_2=n$ and a symmetric cost matrix $C=[c_{ij}]$ with zero diagonal, and positive off-diagonal entries satisfying the triangle inequality $c_{ik}\le c_{ij}+c_{jk}$, the quantity  $\tau(C,P)$ gives rise to a distance between probability vectors $\bp_1$ and $\bp_2$ \cite{Fri20}, which can be viewed as two histograms. 

The entropic relaxation of the classical BPOTP is given by \eqref{EMOT}.
As the Shannon entropy $E_S$ is strictly concave,  the optimization problem \eqref{EMOT} has a unique solution in the relative interior $\rU(P)$.  For simplicity of the exposition we assume in the rest of the paper that $\bp_1,\bp_2>\0$, unless stated otherwise.   By reducing dimensions, if needed,  one can always assume this condition.  See the arguments for the quantum bi-partite case  \cite[Proposition 2.4]{CEFZ23}.  

The following lemma is well known \cite{Fri20} and we bring its proof for completeness:
\begin{lemma}\label{minuecbpotplem}
Let $\bp_i\in\Delta^{n_i},n_i>1,i\in[2]$ be positive probability vectors, and assume that $\rV(P)$ be defined by \eqref{defUPmat}.  
\begin{enumerate}[(a)]
\item 
Consider the minima $\tau(C,P)$ and $\tau(C,P,\varepsilon)$ that are given by \eqref{TOTmat} and \eqref{EMOT} respectively. Then \eqref{bctauin} holds.
In particular for a given $\delta>0$ choose $\varepsilon=\frac{\delta}{\log (n_1 n_2)}$ to obtain the inequality $|\tau(C,P)-\tau(C,P,\varepsilon)|\le \delta$.
\item The unique minimizing matrix $U\in \rV_o(P)$ for \eqref{EMOT} is of the form \eqref{minsolebtot}, where $\bu_1$ and $\bu_2$ are unique.
\end{enumerate}
\end{lemma}
\begin{proof}
(a) Recall that for $\x\in\Delta^n$ we have the inequalities $0\le \rE_S(\x)\le \log n$.
These inequalites yield straightforward (a).

\noindent
(b) The quantity $\tau(C,P,\varepsilon)$ is a constrained minimum problem.   As $-E_S$ is strictly convex, and a directional derivative of $-E_S$ toward the entrior of $\rU(P)$ at any boundary point is $-\infty$ it follows that the unique minimum point $U=[u_{j_1,j_2}]$ is in $\rV_o(P)$.  That is, $U\in\R_{++}^{n_1\times n_2}$.  Recall that $\rV(P)$ is given by $n_1+n_2$ equalities given by \eqref{defUPmat}.  
(Note we have the equality that the sum of the rows are equal to the sum of the columns of $V\in\R^{n_1\times n_2}$.)
Hence, we can use the Lagrange multiplies: $\lambda=(\lambda_1,\ldots, \lambda_{n_1})^\top,\nu=(\nu_1,\ldots,\nu_{n_2})^\top$ corresponding to the row and column sums of $V\in\rV(P)$.   Apply the Lagrange multipliers to deduce \cite{Fri20}
\begin{equation*}
c_{j_1,j_2}+\varepsilon (1+\log u_{j_1,j_2})-\lambda_{j_1}-\mu_{j_2}=0, \quad j_1\in[n_1],j_2\in[n_2].
\end{equation*}
Observe that if we replace $\lambda_{j_1}$ and $\mu_{j_2}$ by  $\lambda_{j_1}+t$ and $\mu_{j_2}-t$ for $j_1\in[n_1],j_2\in[n_2]$ we still have the above equalites.  Thus, we can choose a unique $t$ such that $\sum_{j_1=1}^{n_1}(\lambda_{j_1}-\varepsilon+t)=\sum_{j_2=1}^{n_2}(\mu_{j_2}-t)$, where $u_{j_1,1}=(\lambda_{j_1}-\varepsilon+t)/\varepsilon, u_{j_2,2}=(\mu_{j_2}-t)/\varepsilon$.   This shows that the unique minimal $U$ is of the form \eqref{minsolebtot}.   
It is left to show that $\bu_1,\bu_2$ are unique.   Since the minimal matrix $U=[u_{j_1,j_2}]$ is unique,  and $C$ is known,   we obtain that $\bu_1,\bu_2$ satisfy the equalities:
\begin{equation}\label{u1u2eq}
u_{j_1,1}+u_{j_2,2}=\log u_{j_1,j_2}+c_{j_1,j_2}/\varepsilon=a_{j_1,j_2}, \quad j_1\in[n_1],j_2\in[n_2],
\end{equation}
where $A=[a_{j_1,j_2}]$ is known.  Recall that $\1_{n_1}^\top \bu_1=1_{n_2}^\top \bu_2=s$.  Sum up the $n_1n_2$ equalities in \eqref{u1u2eq} to deduce that $(n_1+n_2)s=\1_{n_1}^\top A\1_{n_2}$.  Sum up the equalites in \eqref{u1u2eq} on either $j_2$ or $j_1$ to obtain
\begin{equation*}
\bu_1=\frac{1}{n_2}(A\1_{n_2}-s\1_{n_1}),  \quad\bu_2=\frac{1}{n_1}(A^\top\1_{n_1}-s\1_{n_2}), \quad s=\frac{\1_{n_1}^\top A\1_{n_2}}{n_1+n_2}.
\end{equation*}
Hence $\bu_1,\bu_2$ are unique.
\end{proof}

The regularization term gives almost linear time approximation $O(n^2)$, ignoring the logarithmic terms, using a variation of the celebrated Sinkhorn algorithm for matrix diagonal scaling \cite{AWR17}.    Recall that the Sinkhorn algorithm is an alternating row scaling to obtain the condition that $U_{2k-1}\1_{n_2}=\bp_1$,  and then  column scaling to obtain the condition $U^\top_{2k}1_{n_1}=\bp_2$ for $k\in\N$, starting with a matrix $U_0\in\R_{++}^{n_1\times n_2}$.

The problem with entropic relaxation is that we need to rescale the matrix $\exp(\circ(-\frac{C}{\varepsilon}))$.  Then,  if $\frac{\min_{j_1\in[n_1],j_2\in[n_2]}|c_{j_1,j_2}|}{\max_{j_1\in[n_1],j_2\in[n_2]}|c_{j_1,j_2}|}>0$ but small, and $\varepsilon$ is small ($=10^{-8})$,  then many entries of $\exp(\circ(-\frac{C}{\varepsilon}))$ would be treated as zero entries in the scaling algorithm.
\subsection{A barrier relaxation of the classical bi-partite optimal transport problem}\label{subsec:brcbot}
A barrier relaxation of the classical BPOTP is given by \ref{betBOTP}, where $\beta(V)$ is given by \eqref{bVP}.  The following lemma is an analog to Lemma \ref{minuecbpotplem}:
\begin{lemma}\label{minubcbpotplem}
Let $\bp_1\in\R^{n_1},\bp_2\in\R^{n_2}$ be positive probability vectors.  Assume that $\rV(P)$ be defined by \eqref{defUPmat}.  
\begin{enumerate}[(a)]
\item 
Consider the minima $\tau(C,P)$ and $\tau_{\beta}(C,P,\varepsilon)$ that are given by \eqref{TOTmat} and \eqref{betBOTP}, respectively. Then
\begin{equation*}\label{btauin}
 \tau(C,P)<\tau_\beta(C,P,\varepsilon)
\end{equation*}
\item The unique minimizing matrix $U\in \rV_o(P)$ for \eqref{betBOTP}, is of the form \eqref{bBPOTus} where $\bu_1$ and $\bu_2$ are unique.
\end{enumerate}
\end{lemma}
\begin{proof} (a) Follows form the observation that $\beta(V)>0$ for $V\in\rV_o(P)$.

\noindent (b) As $\tr (C^\top V)+\varepsilon \beta(V)$ is strictly convex on $\rV_o(P)$
and is $\infty$ on $\partial \rV_o(P)$ it follows that $\tau_{\beta}(C,P,\varepsilon)$ is achieved at the unique point $U=[u_{j_1,j_2}]\in \rV_o(P)$.   Use Lagrange multipliers to deduce
\begin{equation*}
c_{j_1,j_2}-\frac{\varepsilon}{u_{j_1,j_2}}-\lambda_{j_1}-\mu_{j_2}=0
\end{equation*}
As in the proof of Lemma \ref{minuecbpotplem} we can assume that $\sum_{j_1=1}^{n_1}\lambda_{j_1}=\sum_{j_2=1}^{n_2}\mu_{j_2}$.  Hence,  $U$ is of the form  \eqref{bBPOTus}.  To show the uniqueness of $\bu_1,\bu_2$ we use the same arguments as in the proof of part (b) of Lemma \ref{minuecbpotplem}.
\end{proof}

{\it Proof of Theorem \ref{dbcBPOTP}.}   Clearly, the function $\varphi(\x,\y,C)$ is concave in $\hat\rD$.   As $\log x$ is strictly concave on $\R_{++}$,  to show the strict concavity of $\varphi$ on $\rD$, it is enough to show that the affine map 
$T:\rD\to \R_{++}^{n_1\times n_2}$ given by $T(\x,\y)=C-\x\1_{n_1}^\top-\1_{n_2}\y^\top$  is one-to-one.
This is equivalent to the statement 
\begin{equation*}
(\1_{n_1}^\top \x-\1_{n_2}^\top\y=0)\wedge( \x\1_{n_2}^\top+\1_{n_1}\y^\top=0)\Rightarrow \x=\y=0.
\end{equation*}
This implication follows from the proof of part (b) of Lemma \ref{minuecbpotplem},
where we showed the uniqueness of the solution of \eqref{u1u2eq}.

Let  $\bu_1,\bu_2$ be defined in \eqref{bBPOTus}.  We claim that $(\bu_1,\bu_2)$ is a critical point of $\varphi$ in $\hat\rD$.  Indeed:
\begin{equation*}
\begin{aligned}
&\frac{\partial \varphi}{\partial x_{j_1}}(\bu_1,\bu_2)=p_{j_1,1}-\sum_{j_2=1}^{n_2}\frac{\varepsilon}{c_{j_1,j_2} -u_{j_1,1}-u_{j_2,2}}=0, j_1\in[n_1],\\
&\frac{\partial \varphi}{\partial y_{j_2}}(\bu_1,\bu_2)=p_{j_2,2}-\sum_{j_1=1}^{n_2}\frac{\varepsilon}{c_{j_1,j_2} -u_{j_1,1}-u_{j_2,2}}=0, j_2\in[n_1].
\end{aligned}
\end{equation*}
Hence,  $(\bu_1,\bu_2)$ is a maximal point in $\hat\rD$.  As $\1_{n_1}^\top\bu_1=\1_{n_2}^\top \bu_2$, it follows that $(\bu_1,\bu_2)\in\rD$.  As $\varphi$ is strictly concave in $\rD$, it follows that $(\bu_1,\bu_2)$ is the unique maximal point in $\rD$.  Let $(\x^\star,\y^\star)\in \hat\rD\setminus \rD$ be a maximum point in $\hat\rD$.  Then $(\x^\star,\y^\star)+t(\1_{n_1},-\1_{n_2})$ is also a maximal point in $\hat\rD$.  There exists a unique $t^\star$ such that $(\x^\star,\y^\star)+t^\star(\1_{n_1},-\1_{n_2})\in\rD$.  Hence, $(\x^\star,\y^\star)+t^\star(\1_{n_1},-\1_{n_2})=(\bu_1,\bu_2)$.  Therefore,  maximal points of $\varphi$ in $\hat\rD$ are given by \eqref{defLBP}. 

It is left to show the equality \eqref{maxfBOTP}.   Observe first
\begin{equation*}
\begin{aligned}
&\tau_{\beta}(C,P,\varepsilon)=\tr C^\top U +\varepsilon\beta(U)=
\sum_{j_1=j_2=1}^{n_1,n_2}\frac{c_{j_1,j_2}\varepsilon}{c_{j_1,j_2}-u_{j_1,1}-u_{j_2,2}} +\varepsilon\beta(U)=\\
&\varepsilon\left(\sum_{j_1=j_2=1}^{n_1,n_2}\frac{c_{j_1,j_2}-u_{j_1,1}-u_{j_2,2}+(u_{j_1,1}+u_{j_2,2})}{c_{j_1,j_2}-u_{j_1,1}-u_{j_2,2}} \right)+\varepsilon\beta(U)=\\
&n_1n_2\varepsilon+\left(\sum_{j_1=j_2=1}^{n_1,n_2}(u_{j_1,1}+u_{j_2,2})\frac{\varepsilon}{c_{j_1,j_2}-u_{j_1,1}-u_{j_2,2}} \right)+\varepsilon\beta(U)\\
&n_1n_2\varepsilon+\bp_1^\top\bu_1+\bp_2^\top\bu_2-\varepsilon
\sum_{j_1=j_2=1}^{n_1,n_2} \log \frac{\varepsilon}{c_{j_1,j_2}-u_{j_1,1}-u_{j_2,2}}=\\
&\bp_1^\top\bu_1+\bp_2^\top\bu_2+\varepsilon\left(
\sum_{j_1=j_2=1}^{n_1,n_2} \log (c_{j_1,j_2}-u_{j_1,1}-u_{j_2,2})\right)
+\varepsilon n_1n_2(1-\log \varepsilon).
\end{aligned}
\end{equation*}
Observe second
\begin{equation*}
\begin{aligned}
&\max_{(\x,\y)\in\hat\rD}\varphi(\x,\y,C)=\varphi(\bu_1,\bu_2,C)=\\
&\bp_1^\top\bu_1+\bp_2^\top\bu_2+\varepsilon \left(\sum_{j_1=j_2=1}^{n_1,n_2}  \log(c_{j_1,j_2}-u_{j_1,1}-u_{j_2,2}) \right)+\varepsilon n_1n_2(1-\log\varepsilon).
\end{aligned}
\end{equation*}
Hence, the equality \eqref{maxfBOTP}.
\qed
\subsection{The interior point method for finding a barrier approximation of $\tau(C,P)$}\label{subsec:ipmcBP}
In this section we discuss how to apply the IPM method to find an approximation to $\tau(C,P)$ using a modified barrier $\hat\beta$.  A short discussion of the IPM method is given in Appendix \ref{sec:ipm}.  The problem in applying the IPM to the 
\begin{equation}\label{min=maxcBP}
\max_{(\x,\y)\in\rD}\varphi(\x,\y,C)=-\min_{(\x,\y)\in\rD}-\varphi(\x,\y,C)
\end{equation}
is that $\rD$ is unbounded.   We need to restrict the  maximum problem to a bounded domain.   

Let
\begin{equation*}
\begin{aligned}
&\langle \x,\y\rangle=\x^*\y,  \quad \x,\y\in\C^n,\\
&\|\x\|_p=(\sum_{i=1}^n |x_i|^p)^{1/p},  p\in[1,\infty],\quad \x=(x_1,\ldots,x_n)^\top\in\C^n,\\
&\|\x\|=\|\x\|_2,\\
&\langle A,B\rangle=\tr A^*B, \quad \|A\|_F=\sqrt{\langle A,A\rangle},  \quad A,B\in\C^{m\times n}.
\end{aligned}
\end{equation*}
We start with the following lemmas
\begin{lemma}\label{x1xdqfin} Let $2\le d$ be an integer, and $\x_i=(x_{1,i},\ldots,x_{n_i})^\top\in\R^{n_i}, i\in[d]$.  Assume that 
\begin{equation}\label{balancecond}
\1_{n_i}^\top\x_i-\1_{n_{i+1}}^\top\x_{i+1}=0 \textrm{ for }i\in[d-1].
\end{equation}
Then
\begin{equation}\label{sumxjiin}
\begin{aligned}
&\sum_{i=1}^d \|\x_i\|^2\le \frac{\max_{i\in[d]}n_i}{n_1\cdots n_d}
\sum_{j_1=\cdots=j_d=1}^{n_1,\cdots,n_d}(x_{j_1,1}+\cdots+x_{j_d,d})^2\le\\ 
&(\max_{i\in[d]}n_i)\sum_{i=1}^d \frac{d}{n_i}\|\x_i\|^2
\end{aligned}
\end{equation}
\end{lemma}
\begin{proof}  
Clearly,
\begin{equation}
\begin{aligned}
&\sum_{j_1=\cdots=j_d=1}^{n_1,\cdots,n_d}(x_{j_1,1}+\cdots+x_{j_d,d})^2=\\
&\sum_{i=1}^d \frac{n_1\cdots n_d}{n_i}\|\x_i\|^2+2\sum_{1\le p <q\le d}\frac{n_1\cdots n_d}{n_pn_q}(\1_{n_p}^\top\x_{n_p})(\1_{n_q}^\top\x_{n_q})
\end{aligned}
\end{equation}
As $\1_{n_p}^\top\x_{n_p}=\1_{n_q}^\top\x_{n_q}$ we deduce that $(\1_{n_p}^\top\x_{n_p})(\1_{n_q}^\top\x_{n_q})\ge 0$.  Hence,  the first inequality in \eqref{sumxjiin} holds. To deduce the second inequality use the Cauchy-Schwarz inequality: 
$$(x_{j_1,1}+\cdots+x_{j_d,d})^2\le d(x_{j_1,1}^2+\cdots+x_{j_d,d}^2).$$
\end{proof}
\begin{lemma}\label{maxu1u2est}  Let $2\le n_1,n_2\in\N$,  $0\ne C=[c_{j_1,j_2}]\in\R^{n_1\times n_2}$, and $D$ be given by \eqref{DmfbBOTP}.   Denote
\begin{equation}\label{cminmax}
\begin{aligned}
&c_{\min}=\min_{j_1\in[n_1],j_2\in[n_2]}c_{j_1,j_2},\quad  c_{\max}=\max_{j_1\in[n_1],j_2\in[n_2]}c_{j_1,j_2},\\
&c=\max_{j_1\in[n_1],j_2\in[n_2]}|c_{j_1,j_2}|.
\end{aligned}
\end{equation}
\begin{enumerate}[(a)]
\item Consider the minimum problem \eqref{TOTmat}.   
Then
\begin{equation}\label{lbtau}
 \tau(C,P)\ge c_{\min}.
\end{equation}
\item
The dual maximum problem to \eqref{TOTmat} is given by \eqref{dBOTP}.    
\item  Let $\bp_1,\bp_2$ be positive probability vectors. Assume that $C(\x,\y)$ is defined in \eqref{DmfbBOTP}, and suppose that
\begin{equation}\label{xylbin}
(C(\x,\y)\ge 0)\wedge(\bp_1^\top\x+\bp_2^\top\y\ge c_{\min})\wedge(\1_{n_1}^\top\x=\1_{n_2}^\top\y).
\end{equation}
Then
\begin{equation}\label{u1u2ineq}
\begin{aligned}
&0\le c_{j_1,j_2}-x_{j_1}-y_{j_2}<\frac{c_{\max}-c_{\min}}{p_{j_1,1}p_{j_2,2}} \le\\ &\frac{2c}{p_{j_1,1}p_{j_2,2}}\le\frac{2c}{(\min_{j_1\in[n_1]}p_{j_1,1})(\min_{j_2\in[n_2]}p_{j_2,2})},\\
&\sum_{j_1=j_2=1}^{n_1,n_2} (x_{j_1}+y_{j_2})^2<n_1n_2\left(\frac{2c}{(\min_{j_1\in[n_1]}p_{j_1,1})(\min_{j_2\in[n_2]}p_{j_2,2})}\right)^2,\\
&\|(\x,\y)\|_2=\sqrt{\|\bx\|^2+\|\y\|^2}<\\
&\sqrt{\max(n_1,n_2)}\frac{2c}{(\min_{j_1\in[n_1]}p_{j_1,1})(\min_{j_2\in[n_2]}p_{j_2,2})}.
\end{aligned}
\end{equation}
\item Let
\begin{equation}\label{x0y0cBP}
\x_0=\frac{t}{n_1}\1_{n_1}, \y_0=\frac{t}{n_2}\1_{n_2}, \quad t=\frac{n_1n_2(-1+c_{\min})}{n_1+n_2}.
\end{equation}
Then $(\x_0,\y_0)\in\rD$. 
\item  Let 
\begin{equation}\label{defrcBP}
r^2=\max(n_1,n_2)\frac{9c^2+1}{(\min_{j_1\in[n_1]}p_{j_1,1})^2(\min_{j_2\in[n_2]p_{j_2,12})^2})}.
\end{equation}
The interior of the ball
$\rB((\x_0,\y_0),r)=\{\|\x-\x_0\|^2 +\|\y-\y_0\|^2\le  r^2\},$
denoted as $\rB_o((\x_0,\y_0),r)$, contains all $(\x,\y)$ that satisfies the conditoins  \eqref{xylbin}.
\item Let $\rD_1=\rD\cap   \rB_o((\x_0,\y_0),r)$.  Then for any boundary point $(\x,\y)$ of $\rD_1$ the inequality $\|(\x_0,\y_0)-(\x,\y)\|\ge \frac{1}{\sqrt{2}}$ holds.
\item The function
\begin{equation}\label{hbcBT}
\begin{aligned}
&\hat\beta(\x,\y)=-\sum_{j_1=j_2}^{n_1,n_2} \log(c_{j_1,j_2}-x_{i_1}-y_{j_1})\\
&-\log\left(r^2-\|\x-\x_0\|^2-\|\y-\y_0\|^2\right)+\log r^2
\end{aligned}
\end{equation}
is a barrier function for the convex domain $\rD_1$ with $\theta(\hat\beta)\le n_1n_2+1$.
\end{enumerate}
\end{lemma}
\begin{proof}
(a)  Denote 
\begin{equation}\label{defDelbn}
\Delta^{\bn}=\{\cZ=[z_{j_1,\ldots,j_d}]\in (\R^{n_1}\times \cdots\times\R^{n_d})_+, \sum_{j_1=\cdots=j_d=1}^{n_1,\ldots,n_d}z_{j_1,\ldots,j_d}=1\}.
\end{equation} 
Recall that the extreme points of $\Delta^{\bn}$ are tensors which have one entry equal to one, and all other entries are zero.  Clearly,  $\rV(P)\subset \Delta^{(n_1,n_2)}$.
Hence $\tau(C,P)\ge \min_{V\in\Delta^{(n_1,n_2})}\tr C^\top V=c_{\min}$.

\noindent
(b) Consider the following LP problem:
\begin{equation*}
\min_{A\z=\bb, \z\ge \0} \bc^\top \z,  \quad \z,\bc\in\R^n, \bb\in\R^m, A\in\R^{m\times n}.
\end{equation*}
Suppose that $A\z=\bb$ is a bounded nonempty set. Then the set $\bc-A^\top \bw\ge 0$ is nonempty,  and duality holds \cite{CCPS}:
\begin{equation*}
\max_{\bw\in\R^m,\bc-A^\top\bw\ge\0}\bb^\top\bw=\bb^\top\bw^\star=\min_{A\z=\bb, \z\ge \0} \bc^\top \z.
\end{equation*}
Consider the minimum problem \eqref{TOTmat}.   Then $\bb=(\bp_1,\bp_2),\bw=(\x,\y)$  viewed as column vectors in $\R^{n_1+n_2}$.   Thus, $\bb^\top\bw=\bp_1^\top \x+\bp_2^\top \y$.  Observe first  that the condition  $\bc-A^\top\bw\ge 0$ is the condition $C-\x\1_{n_2}^\top-\1_{n_1}\y^\top\ge 0$. 
To show that,  observe the following implications
\begin{equation}\label{complcond}
\begin{aligned}
&(C-\x\1_{n_2}^\top-\1_{n_1}\y^\top\ge 0)\wedge(V\in \rV(P))\Rightarrow\\
&\tr (C-\x\1_{n_2}^\top-\1_{n_1}\y^\top)^\top V\ge 0\Rightarrow
\\
&\tr C^\top V\ge \x^\top\bp_1 +\bp_2^\top \y. 
\end{aligned}
\end{equation}
Furthermore, the duality yields first equality line of \eqref{dBOTP}.  Therefore, the second  equality of \eqref{dBOTP} holds.  Without loss of generality we can replace $(\x,\y)$ by $(\x+t\1_{n_1},\y-t\1_{n_2})$.  Hence,  the third equality of \eqref{dBOTP} holds.  
Furthermore, we can assume that there is a maximum solution $(\bu_1,\bu_2)$  to  \eqref{dBOTP} such that $\1_{n_1}^\top\bu_1=1_{n_2}^\top\bu_2$.

\noindent
(c)  Assume the conditions \eqref{xylbin}.   We have the following implications:
\begin{equation*}
\begin{aligned}
&0\le C(\x,\y)\Rightarrow x_{j_1}+y_{j_2}\le c_{j_1,j_2} \le c_{\max} \textrm{ for all } (j_1,j_2)\in[n_1]\times[n_2],\\
&c_{\min}\le \bp_1^\top \x+  \bp_2^\top \y=\sum_{j_1=j_2=1}^{n_1,n_2} p_{j_1,1}p_{j_2,2}(x_{j_1}+y_{j_2})\le\\
&p_{1,k}p_{2,l}(x_k+y_l)+\sum_{(j_1,j_2)\in[n_1]\times [n_2]\setminus\{(k,l)\}}p_{j_1,1}p_{j_2,2} c_{j_1,j_2}=\\
&p_{1,k}p_{2,l}(x_p+y_q)-p_{1,k}p_{2,l}c_{p,q}+\tr C^\top(\bp_1\bp_2^\top)\Rightarrow\\
&0\le c_{k,l}- x_p-y_q\le \frac{\tr C^\top(\bp_1\bp_2^\top)-c_{\min}}{p_{1,k}p_{2,l}}\le\frac{c_{\max}-c_{min}}{p_{1,k}p_{2,l}}\le\frac{2c}{p_{1,k}p_{2,l}}.
\end{aligned}
\end{equation*}
The last inequality yields the first inequalities in \eqref{u1u2ineq}. Next observe that the above inequalities yield:

\begin{equation*}
\begin{aligned}
&-\frac{2c}{p_{1,k}p_{2,l}}<
\frac{c_{\min}-\tr C^\top(\bp_1\bp_2^\top)+p_{1,k}p_{2,l}c_{k,l}}{p_{1,k}p_{2,l}}\le x_k+y_l\le c_{k,l}\Rightarrow\\
&|x_p+y_q|<\frac{2c}{(\min_{j_1\in[n_1]} p_{j_1,1})(\min_{j_2\in[n_2]}p_{j_2,2})}\forall (p,q)\in[n_1]\times[n_2]\Rightarrow\\
&\sum_{j_1=j_2=1}^{n_1,n_2} (x_{j_1}+y_{j_2})^2< n_1n_2\left(\frac{2c}{(\min_{j_1\in[n_1]} (p_{j_1,1})(\min_{j_2\in[n_2]}p_{j_2,2})}\right)^2.
\end{aligned}
\end{equation*}
The last inequality implies the first inequality of \eqref{u1u2ineq}.   Use Lemma \ref{x1xdqfin} to deduce the second inequality of  \eqref{u1u2ineq}.

\noindent
(d) is straightforward.

\noindent
(e) Assume that $(\x,\y)$ that satisfies the conditoins  \eqref{xylbin}. Observe 
\begin{equation*}
\begin{aligned}
&\|\x-\x_0\|^2+\|\y-\y_0\|^2\le (\|\x\|+\|\x_0\|)^2+(\|\y\|+\|\y_0\|)^2<\\ 
&2(\|\x\|^2+\|\x_0\|^2 +\|\y\|^2 +\|\y_0\|^2)<\frac{2n_1n_2}{n_1+n_2}(-1+c_{\min})^2+\\
&\max(n_1,n_2)\frac{8c^2}{(\min_{j_1\in[n_1]}p_{j_1,1})^2(\min_{j_2\in[n_2]}p_{j_2,2})^2}<\\
&\max(n_1,n_2)\frac{9c^2+1}{(\min_{j_1\in[n_1]}p_{j_1,1})^2(\min_{j_2\in[n_2]}p_{j_2,2})^2} =r^2.
\end{aligned}
\end{equation*}
Hence, $(\x,\y)\in\rB_o((\x_0,\y_0),r)$. 

\noindent
(f) Suppose that $(\x,\y)$ is a boundary point of $\rD_1$.  As $r>1$ it is enough to consider a boundary point of $\rD_1$ which is a boundary point of $D$.  This means that $c_{j_1,j_2}=x_{j_1}+y_{j_2}$ for some $j_1\in [n_1], j_2\in[n_2]$. Clearly,
\begin{equation*}
\begin{aligned}
&2(\|\x-\x_0\|^2+\|\y-\y_0\|^2)\ge 2(|x_{j_1}-t/n_1|^2+|y_{j_2}-t/n_2|^2)\ge\\ 
&(|x_{j_1}-t/n_1|+|y_{j_2}-t/n_2|)^2\ge 
|x_{j_1}+y_{j_2}-\frac{(n_1+n_2)t}{n_1n_2}|^2=\\
&|c_{j_1,j_2}-c_{\min}+1|^2\ge 1.
\end{aligned}
\end{equation*}

\noindent
(g) Clearly,  $\hat\beta$ is a barrier function for $\rD_1$.  Since the complexity  parameter of each term in $\hat\beta$ is bounded above by $1$ we deduce that $\theta(\hat\beta)\le n_1n_2+1$.

\end{proof}

\begin{theorem}\label{IPMcBP}
Let $\bp_i\in\Delta^{n_i},i\in[2]$ be positive probability vectors.  Consider the minimum \eqref{TOTmat}, which corresponds to the maximum dual problem \eqref{dBOTP}.   Apply the IPM short step interior path algorithm with the barrier \eqref{hbcBT} to the $-\min$ problem as in \eqref{min=maxcBP}, where $r$ is given by
\eqref{defrcBP}.   Let $c=\max_{j_1\in[n_1],j_2\in[n_2]}|c_{j_1,j_2}|$. Choose a  starting point  $(\x_0,\y_0)$ given by \eqref{x0y0cBP}.
The CPA algorithm finds the value $\tau(C,P)$ within precision $\delta>0$ in  
\begin{equation}\label{ipmotmat1}
O\big(\sqrt{n_1n_2+1}\log\frac{r(n_1n_2+1)}{\delta}\big)
\end{equation}
iterations. 
For $n_1=n_2=n$ the arithmetic complexity of CPA is $O\big(n^4\log\frac{n^2r}{\delta}\big)$.
\end{theorem}
\begin{proof} Recall that the maximum \eqref{maxfBOTP} is achieved on $\rD_1$.  
Furthermore, the IPM method needs 
\begin{equation}\label{Renthm}
O\big(\sqrt{\theta(\hat\beta)}\log\big(\frac{\theta(\hat\beta)}{\delta \textrm{sym}((\x_0,\y_0),\rD_1)}\big)\big).
\end{equation}
Here,  $\textrm{sym}((\x_0,\y_0),\rD_1)$ is bounded below by the ratio of the minimum distance  and the maximum distance of $(\x_0,\y_0)$ to the boundary.   Part (e) of Lemma \ref{maxu1u2est} yields that $\textrm{sym}((\x_0,\y_0),\rD_1)\ge \sqrt{2}r$.
This shows that the number of iterations to find the value $\tau(C,P)$ within precision $\delta>0$ is given by \eqref{ipmotmat1}.

Assume that $n_1=n_2=n$.
Each step of the iteration of IPM needs the computation of the gradient and the Hessian of $\hat\beta$.    Clearly, $\frac{\partial \hat\beta}{\partial x_{j_1}}$ and  $\frac{\partial \hat\beta}{\partial y_{j_2}}$ have $n+1$ summands.    Therefore, the computation of $\nabla \hat\beta$ is $O(n^2)$.  The diagonal entry of the Hessian $H(f)$ has $n+2$ summands.  Hence, the computation of the diagonal entries of $H(f)$ needs $O(n^2)$ flops.  An off-diagonal entry of $H(f)$ has 2 summands.  Hence, the computation of the off-diagonal entries of $H(f)$ is $O(n^2)$.
To compute $H^{-1}(\hat\beta)(\nabla\beta-(\bp_1,\bp_2))$ we need $O(n^3)$ flops.
Thus, one iteration of the Newton methos needs $O(n^3)$ flops.  As IPM needs $O(n\log\frac{n^2r}{\delta})$ iteration, the arithmetic complexity of the IPM method is $O(n^4\log\frac{n^2r}{\delta})$.
\end{proof}
\subsection{Generalized Sinkhorn algorithm}\label{subsec:GSa}
The following proposition gives a generalization of the inequality \eqref{bctauin}:
\begin{proposition}\label{tauepsest}  Let $2\le n_1,n_2\in\N$, and $0\ne C=[c_{j_1,j_2}]\in\R^{n_1\times n_2}$.   Set $c=\max_{j_1\in[n_1],j_2\in[n_2]}|c_{j_1,j_2}|$.  Assume that $(\bp_1,\bp_2)\in \R^{n_1}\times \R^{n_2}$ be positive probablity vectors, and $\varepsilon>0$.  Then 
\begin{equation}\label{bbtauin}
\begin{aligned}
&\tau(C,P)<\tau_{\beta}(C,P, \varepsilon)<\tau(C,P) +n_1n_2\varepsilon(1-\log\varepsilon)+\\
&n_1n_2\varepsilon\log \left(\frac{2c}{(\min_{j_1\in[n_1]}p_{j_1,1})(\min_{j_2\in[n_2]}p_{j_2,2})}\right).
\end{aligned}
\end{equation}
\end{proposition}
\begin{proof}  The first inequality in \eqref{bbtauin} is stated in part (a) of Lemma \ref{minubcbpotplem}.
The equality \eqref{maxfBOTP} in
Theorem \ref{dbcBPOTP} yields that  
\begin{equation*}
\tau_{\beta}(C,P,\varepsilon)=
\bp_1^\top\bu_1+\bp_2^\top\bu_2-\varepsilon\beta_C(\bu_1,\bu_2) +n_1n_2\varepsilon(1-\log\varepsilon), 
\end{equation*}
where $(\bu_1,\bu_2)$ is a unique point in $\rD$.   The dual characterization of $\tau(C,P)$ given by \eqref{dBOTP} yields that $\bp_1^\top\bu_1+\bp_2^\top\bu_2\le \tau(C,P)$.  Use the first set of inequalities in \eqref{u1u2ineq} to deduce
\begin{equation*}
\begin{aligned}
&\tau_{\beta}(C,P)-\tau(C,P)- n_1n_2\varepsilon(1-\log\varepsilon)\le -\varepsilon \beta_C(\bu_1,\bu_2)=\\
&\varepsilon\log \prod_{j_1=j_2=1}^{n_1,n_2}(c_{j_1,j_2}-u_{j_1,1}-u_{j_2,2})\le \\
&\varepsilon n_1n_2\log
\left(\frac{2c}{(\min_{j_1\in[n_1]}p_{j_1,1})(\min_{j_2\in[n_2]}p_{j_2,2})}\right).
\end{aligned}
\end{equation*}
Finally,  we deduce the inequality
\begin{equation*}
\begin{aligned}
&\tau_{\beta}(C,P)-\tau(C,P)- n_1n_2\varepsilon(1-\log\varepsilon)<\\
&n_1n_2\varepsilon\log \left(\frac{2c}{(\min_{j_1\in[n_1]}p_{j_1,1})(\min_{j_2\in[n_2]}p_{j_2,2})}\right)
\end{aligned}
\end{equation*}
This proves the second inequality in \eqref{bbtauin}.
\end{proof}
In this subsection we need the following lemma:
\begin{lemma}\label{zfalem} Let $2\le n\in\N$ and assume that $\ba=(a_1,\ldots,a_n),a_1\ge\ldots\ge a_n>0$,  and $a_1>a_n$.   (The case $a_1=\cdots=a_n$ is trivial).  Define
\begin{equation}\label{deffba}
f(x, \ba)=\sum_{i=1}^n \frac{1}{x+a_i}, \quad x>-a_n.
\end{equation}
\begin{enumerate} [(a)]
\item The following equalities hold for $x>-a_n$:
\begin{equation}\label{f'f''}
\begin{aligned}
&f'(x,\ba)=-\sum_{i=1}^n\frac{1}{(x+a_i)^2},\\
&f''(x,\ba)=\sum_{i=1}^n\frac{2}{(x+a_i)^3}.
\end{aligned}
\end{equation}
The function $f$ is strictly convex, and $f,-f',f''$ are strictly decreasing from $\infty$ to $0$ on $(-a_n,\infty)$.
\item
Let $a>0$ be given and consider the equation $f(x)-a=0$.  It has a unique solution in $(-a_n,\infty)$, denoted as $x(a)$, which satisfies the following inequalities:
\begin{equation}\label{fxain}
\begin{aligned}
&x_0=\frac{1}{a}-a_n< x(a)< y_0=\frac{n}{a}-a_n,\\
&f(x_0)>0> f(y_0).
\end{aligned}
\end{equation}
\item The point $x(a)$ is the unique minimum of the function $g(x)=ax-\sum_{i=1}^n \log (x+a_i)$ for $x>-a_n$.
\item Consider the Newton and Regula Falsi iterations for $x(a)$:
\begin{equation}\label{NRFit}
\begin{aligned}
&x_{k+1}=x_{k}+\frac{a-f(x_{k})}{f'(x_{k})}, \\
&y_{k+1}=\frac{f(x_k)-a)y_k+(a-f(y_k))x_k}{f(x_k)-f(y_k)},
\end{aligned}
\end{equation}
for $k\ge 0$.  Then $x_{k}$ and $y_k$ are increasing and decreasing sequences that converge to $x(a)$.  Furthermore,  the sequence $\{x_k\}$ and $\{y_k\}$ converges quadratically to $x(a)$.
\item Consider the bisection algorithm for $k\ge 1$:
\begin{equation}\label{bacBP}
\begin{aligned}
&z_k=\frac{1}{2}(x_{k-1}+y_{k-1}),\\
&(x_k,y_k)=\begin{cases}
(z_k, y_{k-1}) \textrm{ if }  f(z_k)>a,\\
(x_{k-1}, z_k)  \textrm{ if }  f(z_k)<a,\\
(z_k,z_k) \textrm{ if } f(z_k)=a,  \textrm{ stop }.
\end{cases}
\end{aligned}
\end{equation}
Then 
\begin{equation}\label{fxkykcBPin}
\begin{aligned}
&0\le f(x_k)-f(y_k)\le -(y_k-x_k)f'(x_k)\le \frac{n(y_k-x_k)}{(x_k+a_n)^2}\le\\ &na^2(y_k-x_k)\le n a^22^{-k}(y_0-x_0)=n (n-1)a 2^{-k}.
\end{aligned}
\end{equation}
Let $\delta>0$ then  
\begin{equation}\label{delercBP}
|f(x_k)-a|,|f(y_k)-a|\le  \delta \textrm{ for }k=\lceil\frac{\log n(n-1)a}{\delta}\rceil.
\end{equation}
Hence, the arithmetic complexity the bisection algorithm for finding $x>-a_n$ such that $|f(x)-a|\le \delta$ is  $O(n\log \frac{na}{\delta})$.
\end{enumerate}
\end{lemma}
\begin{proof} (a) Straightforward.

\noindent
(b) Clearly, $\frac{1}{x+a_n}<f(x)< \frac{n}{x+a_n}$ for $x>-a_n$.  This yields \eqref{fxain}.

\noindent
(c) Straightforward.

\noindent
(d)  The inequalities $f(x_0)>0>f(y_0)$, strict convexity of $f$, and the fact that $f$ decreasing yield that $x(a)>x_1>x_0, x(a)<y_1<y_0$.  Use induction to deduce that $x(a)>x_k>x_{k-1}, x(a)<y_k<y_{k-1}$ for $k\in\N$.   The classical  conditions of Fourier yield \cite{Ost73}
$\lim_{k\to\infty} x_k=x(a)$, and the convergence  is  quadratic.  It is  not hard to show that in this case $\lim_{k\to\infty} y_k=x(a)$ and the convergence is quadratic.

\noindent (e) Let $y>x>-a_n$.  Then
\begin{equation*}
\begin{aligned}
&0<f(x)-f(y)=\sum_{i=1}^n \big(\frac{1}{x+a_i}-\frac{1}{y+a_i}\big)=\\
&(y-x)\sum_{i=1}^n \frac{1}{(x+a_i)(y+a_i)}<-(y-x)f'(x)<\frac{(y-x)n}{(x+a_n)^2}
\Rightarrow\\
&f(x)-f(y)<(y-x)na^2 \textrm{ if } x>\frac{1}{a}-a_n.
\end{aligned}
\end{equation*}
This establishes the inequalities \eqref{fxkykcBPin} and \eqref{delercBP}.

Observe that the arithmetic complexity to compute $f(x)$ is $O(n)$. Hence, the arithmetic complexity the bisection algorithm for finding $x>-a_n$ such that $|f(x)-a|\le \delta$ is  $O(n\log \frac{na}{\delta})$.
\end{proof}

We remark that the best algorithm to find $x$ that satisfies $|f(x)-a|\le \delta$ is first to apply the bisection algorithm and then the Newton-Ralphson and  Regula Falsi iterations.

We state an analogous result to Menon's result for positive matrices \cite{Men68}, which is corollary of part (b) of Lemma \ref{minucbpotplem}:
\begin{corollary}\label{CanBPmen}
Let $C\in\R^{n_1\times n_2}, \bd_i\in\R_{++}^{n_i}, i\in[2]$.  Assume furthermore that $\1_{n_1}^\top\bd_1=\1_{n_2}^\top\bd_2$.  Then there exists a unique $(\x,\y)\in\rD$ such that  $C(\x,\y)^{-\circ}=[(c_{j_1,j_2}-x_{j_1}-y_{j_1})^{-1}]$ has row sum  column sums $\bd_1^\top$ and $\bd_2$ respectively.
\end{corollary}
\begin{proof} Set $\varepsilon=(\1_{n_1}^\top\bd_1)^{-1}$.  Then $\bp_i=\varepsilon\bd_i,i\in [2]$ are positive probability vectors.  Apply part (b) of Lemma \ref{minubcbpotplem}.
\end{proof}

We now discuss an obvious generalization of the Sinkhorn algorithm for finding the unique $(\x,\y)\in\rD$ satisfying the conditions of Corollary \ref{CanBPmen}.    
First we introduce the rebalance step for $C(\bv,\bw)$ for $(\bv,\bw)\in \hat\rD$:
\begin{equation}\label{rebcBP}
\begin{aligned}
&C_{rb}(\bv,\bw)=C(\x,\y), \quad (\x,\y)\in \rD, \\ 
&\x=\bv+t\1_{n_1},\y=\bw-t\1_{n_2},  t=\frac{\1_{n_2}^\top\bw -\1_{n_1}^\top \bv}{n_1+n_2}.
\end{aligned}
\end{equation}
Observe that  
$$\bp_1^\top\x+\bp_2^\top\y=\bp_1^\top\bv+\bp_2^\top\bw, \quad C(\x,\y)=C(\bv,\bw).$$

Given $\bv,\bw\in\hat D$ we now discuss row and column rescaling corresponding to $C(\bv,\bw)$.  
\begin{enumerate}[(a)]
\item Row rescaling: $C_{row}(\bv,\bw)=C(\bv+\bs,\bw)$: is the unique $\bs\in \R^{n_1}$ such that $(\bv+\bs,\bw)\in \hat\rD$ such that $C(\bv+\bs,\bw)\1_{n_2}=\bd_1$.  (The existence and uniqueness of $\bs$ follows from Lemma \ref{zfalem}.)
\item Column rescaling: $C_{col}(\bv,\bw)=C(\bv,\bw+\bt)$: is the unique $\bt\in \R^{n_2}$ such that $(\bv,\bw+\bt)\in \hat\rD$ such that $\1_{n_1}^\top C(\bv,\bw+\bt)=\bd_2^\top$.  (The existence and uniqueness of $\bt$ follows from Lemma \ref{zfalem}.)
\end{enumerate}

\noindent
$\quad\quad$ $\;\;$\textbf{Sinkhorn type algorithm for row and column rescaling} 
\begin{itemize}
\item[] Set $(\x_0,\y_0)\in\rD$
\item[] for $k:=0,1,2, \ldots, $ 
\item[]$\quad$ $C_{row}(\x_k,\y_k)=C(\z_k,\y_k)$
\item[]$\quad$ $C_{col}(\z_k,\y_k)=C(\z_k,\bv_k)$
\item[]$\quad$ $C(\x_{k+1},\y_{k+1})=C_{rb}(\z_k,\bv_k)$
\item[]$\quad\quad$ stop when $\|C(\x_{k+1},\y_{k+1})\1_{n_2}-\bd_1\|<\delta, $
\item[]$\quad\quad\quad\quad\quad$ and $\|1_{n_1}^\top C(\x_{k+1},\y_{k+1})-\bd_2^\top\|<\delta$
\item[]end
\end{itemize} 

The Sinkhorn type algorithm is a coordinate descent (minimization) \cite{LT92,AWR17,TDGU20,LHCJ22,Fri20} for the minimum of the function $-\varphi(\x,\y)$ given in \eqref{mfbBOTP}.  Hence,  the above Sinkhorn type algorithm converges at least linearly.

It seems to us that the minimal solution of the form \eqref{bBPOTus} does not have the problem that the solution \eqref{minsolebtot} has, which was mentioned in the end of subsection \ref{subsec:cbotp}.   As we pointed out the modified Sinkhorn algorithm is $\log n$ longer, than the standard Sinkhorn algorithm.
\section{Classical multi-partite optimal transport problem and its relaxations}\label{sec:cmotp}
In this section we discuss first the  known results on the classical MPOTP and its entropic relaxation.  Second we discuss the barrier relaxations the primary and dual of classical MPOTP.
\subsection{An entropic relaxation of the classical multi-partite optimal transport problem}\label{cbotp}
A multi-partite optimal transport problem (MPOTP) deals with a joint distribution $Z$ of $d$-random variables $X_1,\ldots,X_d$ whose marginals $(\bp_1,\ldots,\bp_d)$ are prescribed.   We assume that $\bp_i\in \Delta^{n_i}$ for $i\in[d]$.  The joint distribution $Z$ is given by a tensor $\cZ=[z_{j_1,\ldots, j_d}] \in (\otimes_{i=1}^d\R^{n_i})_+$,  whose sum of coordinates is one.  Here,  $(\otimes_{i=1}^d\R^{n_i})_+\supset(\otimes_{i=1}^d\R^{n_i})_{++}$ stand for all nonnegative and positive tensors in $\otimes_{i=1}^d \R^{n_i}$ respectively.  Furthermore, $\cZ$ is in the set  
\begin{equation}\label{defV(P)ten}
\begin{aligned}
&\rV(P)=\{\cV=[v_{j_1,\cdots, j_d}]\in (\otimes_{i=1}^d\R^{n_i})_+: \\
&  \sum_{j_k\in[n_k],k\in[d]\setminus\{i\}}v_{j_1,\cdots ,j_d}=p_{j_i,i}, j_i\in[n_i], i\in[d]\},\\
&P=(\bp_1,\ldots,\bp_d).
\end{aligned}
\end{equation}

Denote by $\langle \cC,\cV\rangle$ the standard inner product in $\otimes _{k=1}^d \R^{n_k}$
\begin{equation*}
\langle \cC,\cV\rangle=\sum_{j_k\in[n_k],k\in[d]} c_{j_1,\ldots,j_d}v_{j_1,\ldots,j_d}, \quad \cC,\cV, \in  \otimes _{i=1}^d \R^{n_i}.
\end{equation*}

The MPOTP is given by
\begin{eqnarray}\label{MTOT}
\tau(\cC,P)=\min_{\cV\in \rV(P)}\langle \cC, \cV\rangle.
\end{eqnarray}
It is a LP problem with $\prod_{k=1}^d n_k$ nonnegative variables and $1+\sum_{k=1}^d (n_k-1)$ constraints. 
The MPTOT was considered in \cite{Pi68,Po94} in the context of multidimensional assignment problem, where the entires of the tensor $\cV$ are either $0$ or $1$.
There is a vast literature on  the MPOTP.
See  for example \cite{FV18,BCN19,TDGU20,HRCK,LHCJ22, Fri20,Fri23} and the references therein.   The entropic relaxation of MPOTP can be stated as follows.  For a given $\varepsilon>0$ consider the minimum
\begin{equation}\label{ETOT}
\begin{aligned}
&\tau(\cC,P,\varepsilon)= \min_{\cV\in \rV(P)}\langle \left(\cC, \cV\rangle  -\varepsilon\rE_S(\cV) \right),\\
&\rE_S(\cV)=-\sum_{i_1=\cdots=i_d=1}^{n_1,\ldots,n_d}v_{i_1,\ldots,i_d}\log v_{i_1,\ldots,u_d}.
\end{aligned}
\end{equation}
Recall \cite{Fri20}
\begin{equation}\label{intaueps}
\tau(\cC,P)-\varepsilon \log (n_1\cdots n_d)\le \tau(\cC,P,\varepsilon)\le\tau(\cC,P).
\end{equation}
Using Lagrange multipliers one can show that one needs to find the rescaling factors 
$\lambda_i=(\lambda_{1,i},\ldots,\lambda_{n_i,i})\in\R^{n_i},i\in[d]$ such that the minimal tensor $\cU^{\star}$ is of the form \cite{Fri20}:
\begin{equation}\label{minclsol}
\cU^{\star}=[\exp(\lambda_{j_1,1}+\cdots+\lambda_{j_d,d}-\frac{1}{\varepsilon} c_{j_1,\ldots,j_d})]\in \rV(P).
\end{equation}
The rescaling factors are unique if we assume that 
\begin{equation}
\sum_{i_1=1}^{n_1}\lambda_{j_1,1}=\ldots=\sum_{j_d}^{n_d}\lambda_{j_d,d}.
\end{equation}
The Sinkhorn algorithm consists of choosing in the step $k$ of the iteration a mode $i_k\in[d]$, and rescaling the tensor tensor $\cV_k=[u_{j_1,\ldots,i_d,k}]$ to $\cV_{k+1}=[e^{\lambda_{i_j,i_k}}u_{i_1,\ldots,i_d,k}]$ such that it $i_k$ marginal is $\bp_{i_k}$.
The entropic relaxation was successfully applied in \cite{BCN19,TDGU20,HRCK,LHCJ22, Fri20}.
\subsection{A barrier relaxation of the classical multi-partite optimal transport problem}\label{subsec:brcmot}
The barrier for $(\otimes_{j=1}^d \R^{n_j})_{++}$ is
\begin{equation}\label{bpten}
\beta(\cV)=-\sum_{i_1=\ldots=i_d=1}^{n_1,\ldots,n_d}\log v_{i_1,\ldots,i_d}, \quad \cV=[v_{i_1,\ldots, i_d}]\in (\otimes_{j=1}^d \R^{n_j})_{++}.
\end{equation}
A barrier relaxation of the classical MPOTP is:
\begin{equation}\label{betMOTP}
\tau_{\beta}(\cC,P,\varepsilon)=\min_{\cV\in \rV_o(P)} \langle \cC,\cV\rangle+\varepsilon \beta(V).
\end{equation}

Denote
\begin{equation}\label{balRp}
\begin{aligned}
&(\R^{n_1}\times\cdots\times\R^{n_d})_{++,b}=\\
&\{(\x_1,\ldots,\x_d): \x_j\in\R^{n_j}_{++}, j\in[d], \1_{n_1}^\top\x_1=\cdots=\1_{n_d}^\top\x_d\}.
\end{aligned}
\end{equation}
The following lemma is an extension to Lemma \ref{minubcbpotplem} to $d\ge 2$ modes:
\begin{lemma}\label{minubcmpotplem}
Let $\bp_i\in\Delta^{n_i}, i\in[d]$.  Assume that $\rV(P)$ be defined by \eqref{defV(P)ten}.  
\begin{enumerate}[(a)]
\item 
Consider the minima $\tau(\cC,P)$ and $\tau_{\beta}(\cC,P,\varepsilon)$ that are given by \eqref{MTOT} and \eqref{betMOTP}, respectively. Then
\begin{equation*}\label{btauin}
 \tau(C,P)<\tau_\beta(C,P,\varepsilon)
\end{equation*}
\item The unique minimizing matrix $U\in \rV_o(P)$ for \eqref{betMOTP}  is of the form 
\begin{equation}\label{bMPOTus}
\begin{aligned}
&\cU=[u_{j_1,\ldots,j_d}]\in \rV_o(P),  \\
&u_{j_1,\ldots,j_d}=\frac{\varepsilon}{c_{j_1,\ldots,j_d}-u_{j_1,1}-\cdots-u_{j_d,d}}, \\
&\1_{n_1}^\top\bu_1=\cdots=\1_{n_d}^\top\bu_d,
\end{aligned}
\end{equation}
 where $\bu_1, \ldots,\bu_d$ are unique.
\end{enumerate}
\end{lemma}

The proof of this lemma is a minor modificaton of the proof of Lemma  \ref{minubcbpotplem} and is left to the reader.   Denote
\begin{equation}\label{defC(x1,xd)}
\begin{aligned}
&\cC(\x_1,\ldots,\x_d)=\cC-\sum_{i=1}^d (\otimes_{k=1}^{i-1}\1_{n_k})\otimes\x_i\otimes(\otimes_{k=i+1}^d \1_{n_k}),\\
&\hat\rD=\{(\x_1,\ldots,\x_d)\in\R^{n_1}\times\cdots\times\R^{n_d}: \cC(\x_1,\ldots,\x_d)>0\},\\
&\rD=\{(\x_1,\ldots,\x_d)\in\hat\rD: \1_{n_1}^\top\x_1=\cdots=\1_{n_d}^\top\x_d\}.
\end{aligned}
\end{equation}

The following lemma is an extension of Lemma \ref{maxu1u2est} to $d\ge 2$ modes.  For simplicity of exposition we consider the case $n_1=\cdots=n_d=n$:
\begin{lemma}\label{maxu1udest}  Let $2\le n\in\N$, and $0\ne \cC=[c_{j_1,\ldots,j_d}]\in\otimes^d\R^{n}$.   Denote
\begin{equation}\label{cminmax}
\begin{aligned}
&c_{\min}=\min_{j_i\in[n_i],i\in[d]}c_{j_1,\ldots,j_d},\quad  c_{\max}=\max_{j_i\in[n_i],i\in[d]}c_{j_1,\ldots,j_d},\\
&c=\max_{j_i\in[n_i],i\in[d]}|c_{j_1,\ldots,j_d}|.
\end{aligned}
\end{equation}
\begin{enumerate}[(a)]
\item Consider the minimum problem \eqref{betMOTP}.   
Then
\begin{equation}\label{lmtau}
 \tau(\cC,P)\ge c_{\min}.
\end{equation}
\item
The dual maximum problem to \eqref{betMOTP} is 
\begin{equation}\label{dMOTP}
\begin{aligned}
&\tau(\cC,P)=\max_{(\x_1,\ldots,\x_d)\in\R^{n_1}\times\cdots \times\R^{n_d}, C(\x_1,\ldots,\x_d)\ge 0}\sum_{i=1}^d \bp_i\x_i=\\
&\sup_{(\x_1,\ldots,\x_d)\in\hat\rD} \sum_{i=1}^d \bp_i\x_i=
\sup_{(\x_1,\ldots,\x_d)\in\rD} \sum_{i=1}^d \bp_i\x_i.
\end{aligned}
\end{equation}  
\item  Let $\bp_i, i\in[d]$ are positive probability vectors. Assume that $\cC(\x_1,\ldots,\x_d)$ is defined in \eqref{defC(x1,xd)}, and suppose that
\begin{equation}\label{x1xdlbin}
(C(\x_1,\ldots,\x_d)\ge 0)\wedge(\sum_{i=1}^d \bp_i^\top \x_i\ge c_{\min})\wedge(\1_{n}^\top\x_1=\cdots=\1_{n}^\top\x_d).
\end{equation}
Then
\begin{equation}\label{u1udineq}
\begin{aligned}
\|(\x_1,\cdots,\x_d)\|_2=\sqrt{\sum_{i=1}^d\|\bx_i\|^2}<
\frac{2c\sqrt{n}}{\prod_{i=1}^d\min_{j_i\in[n]}p_{j_i,i}}.
\end{aligned}
\end{equation}
\item Let
\begin{equation}\label{x0y0cBPd}
\x_{i,0}=\frac{t}{n}\1_{n}, i\in[d], \quad t=\frac{(-1+c_{\min})n}{d}.
\end{equation}
Then $(\x_{1,0},\ldots,\x_{d,0})\in\rD$. 
\item  Let 
\begin{equation}\label{defrcBPd}
r^2=\frac{(9c^2+1)n}{(\prod_{i=1}^d\min_{j_i\in[n]}p_{j_i,i})^2}.
\end{equation}
The interior of the ball
$\rB((\x_{1,0},\ldots,\x_{d,0}),r)=\{\sum_{i=1}^d \|\x_i-\x_{i.0}\|^2\le  r^2\},$
denoted as $\rB_o((\x_{1,0},\ldots,\x_{d,0}),r)$, contains $(\x_1,\ldots,\x_d)$ that satisfies the conditoins  \eqref{x1xdlbin}.
\item Let $\rD_1=\rD\cap   \rB_o((\x_{1,0},\ldots,\x_{d,0}),r)$.  Then for any boundary point $(\x_1,\ldots,\x_d)$ of $\rD_1$ the inequality $\|(\x_{1,0},\ldots,\x_{d,0})-(\x_1,\ldots,,\x_d)\|\ge \frac{1}{\sqrt{d}}$ holds.
\item The function
\begin{equation}\label{hbcMT}
\begin{aligned}
&\hat\beta(\x_1,\ldots,\x_d)=-\sum_{j_1=\cdots=j_d=1}^{n} \log(c_{j_1,\cdots,j_d}-x_{j_1,1}-\cdots-x_{j_d,d})\\
&-\log\left(r^2-\sum_{i=1}^d\|\x_i-\x_{0,i}\|^2\right)+\log r^2
\end{aligned}
\end{equation}
is a barrier function for the convex domain $\rD_1$ with $\theta(\hat\beta)\le n^d+1$.
\end{enumerate}
\end{lemma}
The proof of this lemma is similar to the proof of Lemma \ref{maxu1u2est}, and we omit it. The following theorem is a generalization of Theorem \ref{IPMcBP}.
\begin{theorem}\label{IPMcMP}
Let $d\ge 3$, and $\bp_i\in\Delta^{n},i\in[d]$ be positive probability vectors.  Consider the minimum \eqref{betMOTP}, which corresponds to the maximum dual problem \eqref{dMOTP}.   Apply the IPM short step interior path algorithm with the barrier \eqref{hbcMT} to the $-\min$ problem as in \eqref{min=maxcBP}, where $r$ is given by \eqref{defrcBPd}.    Let $c=\max_{j_i\in[n_i],i\in[d]}|c_{j_1,\ldots,j_d}|$.
Choose a  starting point  given by \eqref{x0y0cBPd}.
The CPA algorithm finds the value $\tau(C,P)$ within precision $\delta>0$ in  
$O\big(n^{d/2}\log \frac{n^d\sqrt{d}r}{\delta}\big)$
iterations.  The arithmetic complexity of CPA is $O\big(dn^{3d/2}\log\frac{n^d\sqrt{d}r}{\delta}\big)$.
\end{theorem}
\begin{proof}
Most of the  proof of this theorem is similar to the proof of Theorem \ref{IPMcBP}.  We discuss briefly the number of IPM iterations and the number of flops.  Lemma \ref{maxu1udest} shows that $\frac{r_{out}}{r_{in}}\le \sqrt{d}r$.  Furthermore, $\theta(\hat\beta)\le n^d+1$.   Hence, the number of iterations of IPM is $O\big(n^{d/2}\log \frac{n^d\sqrt{d}r}{\delta}\big)$.

Clearly, $\frac{\partial \hat\beta}{\partial x_{j_k}}$ has $n^{d-1}+1$ summands.    Therefore, the computation of $\nabla \hat\beta$ is $O(dn^d)$.  The diagonal entry of the Hessian $H(f)$ has $n^{d-1}+2$ summands.  Hence, the computation of the diagonal entries of $H(f)$ needs $O(dn^d)$ flops.  An off-diagonal entry of $H(f)$ has 2 summands.  Hence, the computation of the off-diagonal entries of $H(f)$ is $O(dn^2)$.
To compute $H^{-1}(\hat\beta)(\nabla\beta-(\bp_1,\ldots,\bp_2))$ we need $O((dn)^3)$ flops. As  $d\ge 3$, one iteration of the Newton methos needs $O(dn^{d})$ flops.  Thus,  arithmetic complexity of the IPM method is  $O\big(n^{3d/2}\log\frac{ n^d\sqrt{d}r}{\delta}\big)$.
\end{proof}
\section{Quantum bi-partite optimal control and its relaxations}\label{sec:qbotp}
Let $\cH$ be an dimensional vector space with the inner product $\langle \x,\y\rangle$.
Let $\rB(\cH)\supset \rS(\cH)\supset \rS_{+}(\cH)\supset \rS_{++}(\cH)$ be the space of linear operators, the real subspace of self-adjoint operators, the cone of positive semidefinite operators,  and the open set of positive definite operators on $\cH$ respectively.
By choosing a fixed orthonormal basis $|0\rangle,\ldots,|n-1\rangle$ we identify $\cH$  with $\C^n$, where $\langle\x,\y\rangle=\x^*\y$.  Then the above linear operators on $\cH$ correspond to $\C^{n\times n}\supset \rH_n\supset \rH_{n,+}\supset\rH_{n,++}$,
corresponding to the algebra of $n\times n$ complex valued matrices, the Hermitian matrices, the positive semidefinite matrices, and the positive definite matrices respectively.

\subsection{A barrier relaxation of the quantum bi-partite optimal transport problem}\label{subsec:brqbot}
Observe
\begin{equation*}
\nabla \det Z=\adj(Z^\top),  \quad Z\in\C^{n\times n},
\end{equation*}
where we view $\det Z$ as a homogeneous polynomial of degree $n$ in $n^2$ variables.   Suppose that $\det Z$ is not zero, and choose a branch of $\log\det Z$.
Then
\begin{equation}\label{derlogdet}
\nabla \log\det Z=\frac{1}{\det Z} (\adj Z^\top) =(Z^\top)^{-1}.
\end{equation}
If $Z\in\rH_{n,++}$ we assume that $\log\det Z\in\R$.

A Hermitian matrix $H\in \rH_{(n_1,n_2)}$ is viewed as a $4$-tensor with entries $h_{(i_1,j_1)(i_2,j_2)}, i_1,i_2\in [n_1],j_1,j_2\in[n_2]$.  Furthermore, one has the following equalities
\begin{equation*}
\begin{aligned}
&h_{(i_2,j_2))(i_1,j_1)}=\overline{h_{(i_1,j_1)(i_2,j_2)}}, i_1,i_2\in[n_1],j_1,j_2\in[n_2],\\
&H^\top=[\tilde h_{(i_1,j_1)(i_2,j_2)}], \tilde h_{(i_1,j_1)(i_2,j_2)}= h_{(i_2,j_2)(i_1,j_1)},\\
&\tr_1 H=[h_{j_1,j_2,2}],  \quad h_{j_1,j_2,2}=\sum_{i_1=1}^{n_1} h_{(i_1,j_1)(i_1,j_2)},\\
&\tr_2 H=[h_{i_1,i_2,1}],  \quad h_{i_1,i_2,1}=\sum_{j_1=1}^{n_2} h_{(i_1,j_1)(i_2,j_1)}.
\end{aligned}
\end{equation*}
\begin{lemma}\label{minubqbpotplem}
Let $\rho_i\in\rH_{n_i,++,1},n_i>1,i\in[2]$  and assume that $\Gamma(R)$ be defined by \eqref{couprs} for $d=2$.  
\begin{enumerate}[(a)]
\item 
Consider the minima $\kappa(C,R)$ and $\kappa_{\beta}(C,R,\varepsilon)$ that are given by \eqref{qotp} and \eqref{qbMOTP}, respectively. Then
\begin{equation}\label{bqtauin}
 \kappa(C,R)<\kappa_\beta(C,R,\varepsilon).
\end{equation}
\item The unique minimizing matrix $\rho^\star\in \Gamma_o(R)$ for \eqref{qbMOTP}, is of the form 
\begin{equation}\label{qbBPOTus}
\rho^\star=\varepsilon (C-U_1\otimes I_{n_2}-I_{n_1}\otimes U_2)^{-1}\in \Gamma_o(R), \quad U_i\in\rH_{n_i}, i\in[2].
\end{equation}
Furthermore,  there exist a unique solution of the above form satisfying $\tr U_1=\tr U_2$.
\end{enumerate}
\end{lemma}
\begin{proof}(a)  As $-\log\det \rho$ is strictly convex in $\Gamma_o(R)$ and is $\infty$ on $\partial \Gamma(R)$, it follows that $\kappa_{\beta}(C,R)$ as achieved at the unique $\rho^{\star}$.   Clearly $\tr C\rho^\star\ge \kappa(C,R)$.
Since $\rho^{\star}\in \Gamma_o(R)$ we have that $\sum_{k=1}^{n_1n_2} \lambda_k(\rho^\star)=1$ we deduce $0<\det \rho^\star<1$.  Hence, $\beta(\rho^\star)>0$.  As $\varepsilon>0$ we obtain the inequality \eqref{bqtauin}.

\noindent
(b) We apply the Lagrange multipliers to give the neccessary conditions for $\rho^\star$.   
Denote the entries of $\rho\in\Gamma(R), \rho_1,\rho_2$ by 
\begin{equation*}
\rho=[\rho_{(i_1,j_1)(i_2,j_2)}], \rho_1=[\rho_{i_1,i_2,1}], \rho_2=[\rho_{j_1,j_2,2}],, i_1,i_2\in[n_1], j_1,j_2\in[n_2].
\end{equation*}
The linear conditions for $\tr_2\rho=\rho_1$ and $\tr_1\rho =\rho_2$ can be expressed as follows:
\begin{equation*}
\begin{aligned}
&-\rho_{i_1,i_2,1}+\sum_{j=1}^{n_2} \rho_{(i_1,j)(i_2,j)}=0, \quad i_1,i_2\in[n_1],\\
&-\rho_{j_1,j_2,1}+\sum_{i=1}^{n_1}\rho_{(i,j_1)(i,j_2)}=0, \quad j_1,j_2\in[n_2].
\end{aligned}
\end{equation*}
Taking in account that $\rho,\rho_1,\rho_2$ are Hermitian, the minimizing function for $\kappa_{\beta}(C,R,\varepsilon)$ with Lagrange mulipliers can be written as
\begin{equation*}
\begin{aligned}
&\left(\sum_{i_1=i_2=j_1=j_2=1}^{n_1,n_1,n_2,n_2} c_{(i_1,j_1)(i_2,j_2)}\rho_{(i_2,j_2)(i_1,j_1)}\right)-\varepsilon \log\det \rho\\
&-\sum_{i_1=i_2=1}^{n_1}\lambda_{i_2,i_1}\left(-\rho_{i_1,i_2,1}+\sum_{j=1}^{n_2} \rho_{(i_1,j)(i_2,j)}\right)\\
&-\sum_{j_1=j_2=1}^{n_2}\mu_{j_2,j_1}\left(-\rho_{j_1,j_2,2}+\sum_{i=1}^{n_1} \rho_{(i,j_1)(i,j_2)}\right),\\
&\Lambda=[\lambda_{i_1,i_2}]\in \rH_{n_1},   M=[\mu_{j_1,j_2}]\in \rH_{n_2}.
\end{aligned}
\end{equation*}
Taking the derivative with respect $\rho_{(i_1,j_1)(i_2,j_2)}$ and setting it to zero, and using \eqref{derlogdet}, we deduce
\begin{equation*}
\begin{aligned}
\varepsilon((\rho^\star)^{-1})_{(i_1,j_1)(i_2,j_2)}=\begin{cases}
c_{(i_1,j_1)(i_2,j_2)} \textrm{ if } i_1\ne i_2 \textrm{ and }j_1\ne j_2,\\
c_{(i,j_1)(i,j_2)}-\mu_{j_1,j_2} \textrm{ if } j_1\ne j_2,\\
c_{(i_1,j)(i_2,j)}-\lambda_{i_1,i_2} \textrm{ if } i_1\ne i_2,\\
c_{(i,j)(i,j)}-\lambda_{i,i}-\mu_{j,j}
\end{cases}\\
\end{aligned}
\end{equation*}
That is, 
$$\varepsilon(\rho^\star)^{-1}=C-\Lambda\otimes I_{n_2}-I_{n_1}\otimes M\Rightarrow \rho^\star=\varepsilon (C-\Lambda\otimes I_{n_2}-I_{n_1}\otimes M)^{-1}.$$
This proves \eqref{qbBPOTus}.   Observe that 
$$U_1\otimes I_{n_2}+I_{n_1}\otimes U_2=(U_1+tI_{n_1})\otimes I_{n_2}+I_{n_1}\otimes (U_2-tI_{n_2}),  \forall t\in\R.$$
Choose $t=\frac{\tr U_2-\tr U_1}{n_1+n_2}$.   Then,  $\tr(U_1+tI_{n_1})=\tr(U_2-tI_{n_2})$. It is left to show that there exists a unique $(U_1,U_2)$ satisfying \eqref{qbBPOTus} satisfying $\tr U_1=\tr U_2$.
This is equivalent to the implication
\begin{equation}\label{U1U2imp}
(\tr U_1=\tr U_2)\wedge (U_1\otimes I_{n_2}+I_{n_1}\otimes U_2=0)\Rightarrow (U_1=0)\wedge(U_2=0).
\end{equation}
Observe that second condition yields that $U_1$ and $U_2$ diagonal.  Furthermore, the second condition yields that $u_{i,i,1}+u_{j,j,2}=0$ for $i\in[n_1], j\in[n_2]$.
As $\tr U_1=\tr U_2$ we deduce as in the proof of part (b) of Lemma \ref{minubcbpotplem} that $U_1=0, U_2=0$.
\end{proof}
{\it Proof of Theorem \ref{dbqBPOTP} .}   As $\log\det Z$ is striclty concave on $\rH_{(n_1,n_2),++}$ it follows that $\varphi(X,Y,C)$ is concave on $\hat\rD$.  To show that $\varphi$ is strictly concave on $\rD$ we need to show that if $\rC(X,Y)=\rC(X_1,Y_1)$ for $(X,Y),(X_1,Y_1)\in\rD$, then $(X,Y)=(X_1,Y_1)$.
This is equivalent to the implication \eqref{U1U2imp}, which is shown in the proof of Lemma \ref{minubqbpotplem}. 

Consider the minimal $\rho^\star\in \Gamma_o(R)$ of the form \eqref{qbBPOTus}.
We claim that that $C(U_1,U_2)$ is a critical point of $\varphi(X,Y,C)$ in $\hat\rD$.  Assume that $X=[x_{i_1,i_2}]\in\rH_{n_1}, Y=[y_{j_1,j_2}], (X,Y)$ in a neighborhood of $(U_1,U_2)$.  The  branch of $\log \det C(X,Y)$ is determined by $\log\det C(U_1,U_2)\in\R$. Then
\begin{equation*}
\begin{aligned}
&\frac{\partial \varphi(X,Y,C)}{\partial x_{i_1,i_2}}=\rho_{i_2,i_1,1}-\varepsilon\sum_{j=1}^{n} ((C(X,Y)^\top)^{-1})_{(i_1,j)(i_2,j)}=\\
&\rho_{i_2,i_1,1}-\varepsilon\sum_{j=1}^{n} (C(X,Y)^{-1})_{(i_2,j)(i_1,j)}, i_1,i_2\in[n_1],\\
&\frac{\partial \varphi(X,Y,C)}{\partial y_{j_1,j_2}}=\rho_{j_2,j_1,2}-\varepsilon\sum_{j=1}^{n} ((C(X,Y)^\top)^{-1})_{(i,j_2)(i,j_1)}=\\
&\rho_{j_2,j_1,2}-\varepsilon\sum_{j=1}^{n} (C(X,Y)^{-1})_{(i,j_2)(i,j_1)}, j_1,j_2\in[n_1].
\end{aligned}
\end{equation*}
Set $(X,Y)=(U_1,U_2)$ and use the equality \eqref{qbBPOTus} to deduce that $(U_1,U_2)$ is a critical point $\varphi(X,Y,C)$ in $\hat\rD$.  As $\varphi$ is concave in $\hat \rD$ it follows that $(U_1,U_2)$ is a maximum point of $\varphi$ in $\hat\rD$.   Let $L$ be the line given by \eqref{defqLBP}.  Then $C(X,Y)$ restricted to $L$ is $C(U_1,U_2)$.  Therefore,  $\varphi(X,Y,C)$ achieves its maximum at $L$.  As in the proof of Lemma \ref{minubqbpotplem} it follows, there exists a unique point such that $L\cap \rD=(U_1^\star,U_2^\star)$. As $\varphi$ is strictly concave on $\rD$ the point $(U_1^\star,  U_2^\star)$ is the unique maximum point of $\varphi$ in $\rD$.  Hence, the set of maximal points of $\varphi$ in $\hat\rD$ is $L$.  

It is left to show that $\varphi(U_1,U_2,C)=\kappa_{\beta}(R,C,\varepsilon)$, which is equivalent to  \eqref{maxqfBOTP}.
Finally, observe
\begin{equation}\label{dkapRC}
\begin{aligned}
&\kappa_{\beta}(R,C,\varepsilon)= 
\tr C\rho^\star-\varepsilon\log\det\rho^\star=\\
&\tr \big(C(U_1,U_2)+U_1\otimes I_{n_2}+I_{n_1}\otimes U_2)\rho^\star-\\
&\varepsilon\log\det\big(\varepsilon C(U_1,U_2)^{-1}\big)=\\
&\tr \big(C(U_1,U_2)\rho^\star +\tr(U_1\otimes I_{n_2}+I_{n_1}\otimes U_2)\rho^\star-\\&\varepsilon\log(\varepsilon^{n_1n_2}(\det C(U_1,U_2)^{-1}))\\
&=\varepsilon\tr I_{n_1}\otimes I_{n_2} +\tr U_1\rho_1+\tr U_2\rho_2-\\
&n_1n_2\varepsilon\log \varepsilon+\varepsilon\log\det C(U_1,U_2)=\\
&\tr \rho_1 U_1+\tr \rho_2 U_2 +\varepsilon\log\det C(U_1,U_2) +\varepsilon n_1n_2(1-\log\varepsilon)=\\
&\varphi(U_1,U_2,C)=\max_{(X,Y)\in\rD}\varphi(X,Y,C).
\end{aligned}
\end{equation}
\qed
\subsection{The interior point method for finding a barrier approximation of $\kappa(C,R)$}\label{subsec:ipmqBP}
In what follows we need the following well known inequalities \cite[Theorem 4.9.1]{Frib}:
\begin{equation}\label{tracein}
 \sum_{j=1}^n \lambda_j(A)\lambda_{n-j+1}(B)\le\tr AB\le \sum_{j=1}^n \lambda_j(A)\lambda_j(B) \textrm{ for } A,B\in\rH_n.
\end{equation}
In particular, one has the following inequality
\begin{equation}\label{tracein}
\lambda_{\min}(A)\le \tr AB\le \lambda_{\max}(A) \textrm{ for } A\in\rH_n, B\in\rH_{n,+,1}.
\end{equation}
The following lemma is an analog of Lemma \ref{maxu1u2est}: 
\begin{lemma}\label{maxqbest}
Let $2\le n_1,n_2\in\N$, and $0\ne C\in\rH_{(n_1,n_2)}$.   
\begin{enumerate}[(a)]
\item Consider the minimum problem \eqref{qotp}.
Then
\begin{equation}\label{lbtau}
 \kappa(C,R)\ge \lambda_{\min}(C).
\end{equation}
\item Assume that $C(X,Y), \hat \rD,\rD$ are defined in \eqref{defqDhD}.   
The dual maximum problem to \eqref{qotp} is given by 
\begin{equation}\label{dqBOTP}
\begin{aligned}
&\kappa(C,R)=\sup_{(X,Y)\in\rH_{n_1}\times \rH_{n_2}, C(X,Y)\succeq 0}\tr\rho_1 X+\tr\rho_2 Y=\\
&\sup_{(X,Y)\in\hat\rD} \tr\rho_1 X+\tr\rho_2 Y=
\sup_{(X,Y)\in\rD} \tr\rho_1 X+\tr\rho_2 Y.
\end{aligned}
\end{equation} 
\item  Let $\rho_i\in\rH_{n_i,++,1}, i\in[2]$.   Suppose that
\begin{equation}\label{XYlbin}
(C(X,Y)\succeq 0)\wedge(\tr\rho_1 X+\tr\rho_2 Y\ge \lambda_{\min}(C))\wedge(\tr X=\tr Y).
\end{equation}
Then
\begin{equation}\label{XYineq}
\begin{aligned}
&\|(X,Y)\|_F=\sqrt{\|X\|_F^2+\|Y\|_F^2}<
\sqrt{\max(n_1,n_2)}\frac{2\|C\|_2}{\lambda_{\min}(\rho_1)\lambda_{\min}(\rho_2)},\\
&-\frac{2\|C\|_2}{\lambda_{\min}(\rho_1)\lambda_{\min}(\rho_2)}\le x_k+y_l\le \|\C\|_2.
\end{aligned}
\end{equation}
\item Let
\begin{equation}\label{X0Y0qBP}
X_0=\frac{t}{n_1}I_{n_1}, Y_0=\frac{t}{n_2}I_{n_2}, \quad t=\frac{n_1n_2(-1+\lambda_{\min}(C))}{n_1+n_2}.
\end{equation}
Then $(X_0,Y_0)\in\rD$. 
\item  Let 
\begin{equation}\label{defrqBP}
r^2=\max(n_1,n_2)\frac{9\|C\|_2^2+1}{(\lambda_{\min}(\rho_1)\lambda_{\min}(\rho_2))^2}.
\end{equation}
The interior of the ball
$\rB((X_0,Y_0),r)=\{\|X-X_0\|_F^2 +\|Y-Y_0\|_F^2\le  r^2\},$
denoted as $\rB_o((X_0,Y_0),r)$, contains all $(X,Y)$ that satisfies the conditoins  \eqref{XYlbin}.  In particular,  the three suprema of \eqref{dqBOTP} are achieved. 
\item Let $\rD_1=\rD\cap   \rB_o((X_0,Y_0),r)$.  Then for any boundary point $(X,Y)$ of $\rD_1$ the inequality $\|(X_0,Y_0)-(X,Y)\|_F\ge \frac{1}{\sqrt{2}}$ holds.
\item The function
\begin{equation}\label{hbqBT}
\begin{aligned}
&\hat\beta(X,Y)=-\log\det C(X,Y)\\
&-\log\left(r^2-\|X-X_0\|_F^2-\|Y-Y_0\|_F^2\right)+\log r^2
\end{aligned}
\end{equation}
is a barrier function for the convex domain $\rD_1$ with $\theta(\hat\beta)\le n_1n_2+1$.
\end{enumerate}
\end{lemma}
\begin{proof} (a) follows from \eqref{tracein}.

\noindent
(b) The first equality in \eqref{dqBOTP} is proved in \cite[Theorem 3.2]{CEFZ23}.
The second equality is straightforward.   Clearly, $C(X,Y)=C(X+tI_{n_1},Y-tI_{n_2})$.
As in the proof of Lemma \ref{minubqbpotplem} choose $t=\frac{\tr U_2-\tr U_1}{n_1+n_2}$ to deduce the third equality.

\noindent (c)  Denote
\begin{equation}\label{defdiagx}
\diag(\x):=\begin{bmatrix}x_1&0&\cdots&0&0\\0&x_2&\cdots&0&0\\\vdots&\vdots&\vdots&\vdots&\vdots\\0&0&\cdots&0&x_n\end{bmatrix}, \quad
\x=(x_1,\ldots,x_n)^\top \in\R^n.
\end{equation}
Assume that $X=\diag(\x),Y=\diag(\y)$.  Denote by $\diag(C)\in\rH_{(n_1,n_2)}$ the diagonal matrix obtained from $C$ by deleting the off-diagonal entries of $C$.
Name the diagonal entry of $\diag(C)$ in the row and column $(j_1,j_2)$ by $c_{j_1,j_2}$, and let $\hat C=[c_{j_1,j_2}]\in\R^{n_1\times n_2}$.   Recall the well known inequalities following from the Cauchy interlacing inequalities \cite[Section 4.4, Problem (3a)]{Frib}:
\begin{equation*}
\begin{aligned}
&\lambda_{\min}(C)\le c_{\min}= \min_{j_1\in[n_1],j_2\in[n_2]} c_{j_1,j_2}\le \max_{j_1\in[n_1],j_2\in[n_2]}
 c_{j_1,j_2}= c_{\max}\le \lambda_{\max}(C)\\
 &\Rightarrow c:=\|C\|_2\ge \max_{j_1\in[n_1],j_2\in[n_2]}|c_{j_1,j_2}|.
 \end{aligned}
 \end{equation*}
 Assume that the conditions of the implication in \eqref{XYlbin} hold for $X=\diag(\x),Y=\diag(\y)$ hold.  
 Restricting these conditions to the diagonal entries of $C(X,Y)$ we obtain the conditions:
 \begin{equation*}
 c_{j_1,j_2}-x_{j_1}-y_{j_2}\ge 0, j_1\in[n_1], j_2\in[n_2],  \bp_1^\top\x+\bp_2^\top\y\ge \lambda_{min}(C),\1_{n_1}^\top\x=\1_{n_2}^\top \y,
 \end{equation*}
 where $\bp_i$ are the diagonal entries of $\rho_i$ for $i\in[2]$.
 Observe that the first set of the above inequalities yield $\tr \hat C^\top \bp_{1}\bp_{2}^\top\ge \bp_{1}^\top\x+\bp_{2}^\top\y\ge \lambda_{min}(C)$. 
 Hence, the following inequaltiy holds:
 \begin{equation*}
\begin{aligned}
&\lambda_{\min}(C)\le \bp_1^\top \x+  \bp_2^\top \y=\sum_{j_1=j_2=1}^{n_1,n_2} p_{j_1,1}p_{j_2,2}(x_{j_1}+y_{j_2})\le\\
&p_{1,k}p_{2,l}(x_k+y_l)+\sum_{(j_1,j_2)\in[n_1]\times [n_2]\setminus\{(k,l)\}}p_{j_1,1}p_{j_2,2} c_{j_1,j_2}=\\
&p_{1,k}p_{2,l}(x_k+y_l)-p_{1,k}p_{2,l}c_{k,l}+\tr C^\top(\bp_1\bp_2^\top).
\end{aligned}
\end{equation*}
Combine above conditions and the inequality to obtain:
\begin{equation}\label{p1p2ineq}
\begin{aligned}
&0\le c_{k,l}- x_k-y_l\le \frac{\tr C^\top(\bp_1\bp_2^\top)-\lambda_{\min}(C)}{p_{1,k}p_{2,l}}\le\\
&\frac{\lambda_{\max}(C)-\lambda_{min}(C)}{p_{1,k}p_{2,l}}\le\frac{2c}{p_{1,k}p_{2,l}}\le \frac{2c}{\lambda_{\min}(\rho_1)\lambda_{\min}(\rho_2)},\\
&-\frac{2c}{p_{1,k}p_{2,l}}<\frac{c_{k,l}p_{1,k}p_{2,l}-\tr C^\top(\bp_1\bp_2^\top)+\lambda_{min}(C)}{p_{1,k}p_{2,l}}\le x_k+y_l.
\end{aligned}
\end{equation}
 The Cauchy interlacing property yields that 
 $\min_{j_i\in[n_i]}\bp_i\ge \lambda_{\min}(\rho_i)$ for $i\in[2]$.  
 This proves the above inequalities.
 Combine the above inequalities to deduce \eqref{XYineq} in this case.
 
 Assume the conditions \eqref{XYlbin}.     Let 
 $$X=Q_X\diag(\lambda(X))Q_X^*, \quad  Y=Q_Y\diag(\lambda(X))Q_Y^*, $$ where $Q_X,Q_Y$ are corresponding unitary matrices.  Denote
 \begin{equation*}
 \rho_1'=Q_X^*\rho_1 Q_X, \quad \rho_2'=Q_Y^*\rho_2 Q_Y, \quad C'=(Q_X^*\otimes Q_Y^*)C(Q_X\otimes Q_Y).
 \end{equation*}
 Observe that $\rho_1',\rho_2',C'$ are cospectral with $\rho_1,\rho_2,C$.  Furthermore,
 $\kappa(C,R)=\kappa(C',R'), R'=(\rho_1',\rho_2')$ \cite[Corollary A.2]{CEFZ23}.   Observe that $X'=\diag(\lambda(X)),  Y'=\diag(\lambda(Y))$ satisfy the conditions 
 \begin{equation*}
 (C'(X',Y)'\succeq 0)\wedge(\tr\rho_1' X'+\tr\rho_2' Y'\ge \lambda_{\min}(C'))\wedge(\tr X'=\tr Y').
 \end{equation*}
 As $X',Y'$ are diagonal matrices we deduce the inequality \eqref{XYineq} from the previous case.
 
 \noindent (d)  Observe that 
\begin{equation}\label{lamCXY0}
\begin{aligned}
&C(X_0,Y_0)=C-\lambda_{\min}(C) I_{n_1}\otimes I_{n_2} +I_{n_1}\otimes I_{n_2}\succeq \\
&I_{n_1}\otimes I_{n_2}\Rightarrow
\lambda_{\min}(C(X_0,Y_0))\ge 1.
\end{aligned}
\end{equation}
(e)  As in the proof of part (c),  considering $\rho_1',\rho_2',C',X',Y'$, where $X',Y'$ are diagonal.  Observe $X_0'=X_0,  Y_0'=Y_0$.  Deduce (e) from part (e) of Lemma \ref{maxu1u2est}.  
 
\noindent (f) 
Observe that  $(X,Y)$ is on the boundary of $D_1$ if either $\|X-X_0\|_F^2+ \|Y-Y_0\|_F^2=r^2$ or $\lambda_{\min}(C(X,Y))=0$.  As $r>1$ it is enough to consider the case $\lambda_{\min}(C(X,Y))=0$.   Use the Weyl's inequalities \cite[Theorem 4.4.6]{Frib} and 
\eqref{lamCXY0} to deduce
\begin{equation*}
\begin{aligned}
&\lambda_{\max}((X-X_0)\otimes I_{n_2}+I_{n_1}\otimes (Y-Y_0))=\\
&\lambda_{\max}((X-X_0)\otimes I_{n_2}+I_{n_1}\otimes (Y-Y_0)) +\lambda_{\min}(C(X,Y))\ge\\ &\lambda_{\min}((X-X_0)\otimes I_{n_2}+I_{n_1}\otimes (Y-Y_0)+C(X,Y))=\\
&\lambda_{\min}(C(X_0,Y_0))\ge 1.
\end{aligned}
\end{equation*}
The arguments in the second part of (c)  yield that the eigenvalues of $(X-X_0)\otimes I_{n_2}+I_{n_1}\otimes (Y-Y_0)$ are $\lambda_{j_1}(X-X_0)+\lambda_{j_2}(Y-Y_0)$ for $j_1\in[n_1],j_2\in[n_2]$.  Hence
\begin{equation*}
\begin{aligned}
&2(\|X-X_0\|^2_F+\|Y-Y_0\|_F^2)\ge 2(\lambda_{\max}^2(X-X_0)+\lambda_{\max}^2(Y-Y_0))\ge\\ 
&(\lambda_{\max}(X-X_0)+\lambda_{\max}(Y-Y_0))^2=\\
&\lambda_{\max}^2((X-X_0)\otimes I_{n_2}+I_{n_1}\otimes (Y-Y_0))\ge 1.
\end{aligned}
\end{equation*}
(g) Recall that $\theta(-\log \det A)=n$ for $A\in\rH_{n,++}$ \eqref{-logdetX}.  Hence, 
\begin{equation*}
\begin{aligned}
&\theta(-\log\det C(X,Y))\le n_1n_2\textrm{ for }(X,Y)\in\rD_1,\\
&\theta(\hat\beta)\le n_1n_2+1.
\end{aligned}
 \end{equation*}
\end{proof}
The following proposition is an analog of \eqref{bqkapin} and \eqref{bbtauin}:
\begin{proposition}\label{kapbepsest}  Let $2\le n_1,n_2\in\N$, and $0\ne C\in\rH_{(n_1,n_2)}$.  Assume that $\rho_i\in\rH_{n_i,++,1}, i\in[2]$ , and $\varepsilon>0$.  Then 
\begin{equation}\label{bbqtauin}
\begin{aligned}
&\kappa(C,R)<\kappa_{\beta}(C,R,\varepsilon)<\kappa(C,R) +n_1n_2\varepsilon(1-\log\varepsilon)+\\
&n_1n_2\varepsilon\log \left(\frac{2\|C\|_2}{\lambda_{n_1}(\rho_1)\lambda_{n_2}(\rho_2)}\right)
\end{aligned}
\end{equation}
\end{proposition}
\begin{proof}  The first inequality in \eqref{bbqtauin} is \eqref{bqtauin}.
Assume that the minimal solution to \eqref{qbMOTP} is given by \eqref{qbBPOTus}, where $\tr U_1=\tr U_2$.  Now use \eqref{dkapRC} to deduce 
\begin{equation*}
\kappa_{\beta}(R,C,\varepsilon)-\varepsilon n_1n_2(1-\log\varepsilon)=\tr \rho_1 U_1+\tr \rho_2 U_2 +\varepsilon\log\det(C(U,V)).
\end{equation*}
By conjugating $U_1,U_2,C$  by approriate unitary matrices, we can assume, as in the proof part (c) of Lemma \ref{maxqbest},  $U_1=\diag(\x), U_2=\diag(\y)$ 
The maximal characterization of $\kappa(C,R)$ given by \eqref{dqBOTP} yields that 
$\tr \rho_1 U_1+\tr \rho_2 U_2\le \kappa(C,R)$.   As $C(U_1,U_2)$ is a positive definite matrix it follows that $\det C(U_1,U_2)$ is at most the product of the diagonal entires of $C(U_1,U_2)$.  (This is Hadamard's determinant inequality.)
As the proof of part (c) of Lemma \ref{maxqbest} denote  the diagonal entries of $C$ in the row and column $(j_1,j_2)$ by $c_{j_1,j_2}$. 
Let $c= \|C\|_2$.
Use the inequalities \eqref{p1p2ineq} to obtain
\begin{equation*}
\begin{aligned}
&\log\det C(U_1,U_2)\le\log\prod_{j_1=j_2}^{n_1,n_2}\big(c_{j_1,j_2}-(x_{j_1}+y_{j_2})\le \\
&\log \big(\frac{2\|C\|_2}{\lambda_{\min}(\rho_1)\lambda_{\min}(\rho_2)}\big)^{n_1n_2}.
\end{aligned}
\end{equation*}
This establishes \eqref{bbqtauin}.
\end{proof}

In what follows we need the following well lemma, whose  first part is well known:
\begin{lemma}\label{loddetder}
Let $2\le n\in\N$, and $Z=[z_{p,q}]\in\C^{n\times n}$.  
\begin{enumerate} [(a)]
\item  Assume that $\|Z\|_2:=\sqrt{\lambda_{\max}(Z^*Z)} <1$.  Then 
\begin{equation}\label{logdetexp1}
-\log\det (I+Z)=\sum_{j=1}^\infty \frac{(-1)^j\tr Z^j}{j}.
\end{equation}
In particular
\begin{equation}\label{logdetex}
\begin{aligned}
&-\log\det(I+Z)=-\left(\sum_{p=1}^n z_{p,p}\right)+\frac{1}{2}\left(\sum_{p=1}^n z_{p,p}^2\right) +\\
&\left(\sum_{1\le p<q\leq  n} z_{p,q}z_{q,p}\right) + \textrm{ higher order homogeneous polynomials}.
\end{aligned}
\end{equation}
If $Z$ is Hermitian, i.e.  $Z=S+\bi T$, where $S=[x_{p,q}], T=[t_{p,q}]\in\R^{n\times n}$ are symmetric and skew symmetric respectively, then
\begin{equation}\label{logdetexher}
\begin{aligned}
&-\log\det(I+Z)=-\left(\sum_{p=1}^n x_{p,p}\right)+\frac{1}{2}\left(\sum_{p=1}^n x_{p,p}^2\right) +\\
&\left(\sum_{1\le p<1\leq  n} x_{p,q}^2+y_{p,q}^2\right) + \textrm{ higher order homogeneous polynomials}.
\end{aligned}
\end{equation}
\item Assume that $E\in\rH_{n,++}$ and $W=[w_{i,j}]\in\rH_n$, such that $\|W\|_2\le \|E^{-1}\|_2^{-1}$.  Then 
\begin{equation}\label{expldEW}
f(W):=-\log\det(E+W)+\log\det E=-\log\det(I+Z),  Z=E^{-1}W
\end{equation}
The complexity of finding the gradient and the Hessian of $f$ at $W=0$ is 
$O(n^3)$ and $O(n^4)$ respectively.
\item Assume that $(A,B)\in \hat\rD$, and let $C(A,B)^{-1}=[e_{(i,p)(j,q)}]\in\rH_{(n_1,n_2),++}$.   Consider the function 
\begin{equation}\label{f(X,Y)}
f(X,Y)=-\log\det(C(A+X,B+Y)), \quad (X,Y)\in \rH_{n_1}\times \rH_{n_2}
\end{equation}
in the neighborhood of $(0,0)$. Set  $Z=-C(A,B)^{-1}(X\otimes I_{n_2}+I_{n_1}\otimes Y)$.  

Then $f(X,Y)=-\log\det C(A, B) -\log\det (I+Z)$, and 
\begin{equation*}
z_{(i,p)(j,q)}=-\sum_{k=1}^{n_1}e_{(i,p)(k,q)}x_{k,j}-\sum_{l=1}^{n_2}e_{(i,p)(j,l)}y_{l,q}.
\end{equation*}
Then the gradient and the quadratic form corresponding to the Hessian of $f$ at $(0,0)$ are
\begin{equation}\label{drasHesXY}
\begin{aligned}
&\nabla f(X,Y)=(F,G)\in \rH_{n_1,++}\times \rH_{n_2,++}, \\
 &F=[f_{k,i}], f_{k,i}=\sum_{p=1}^{n_2} e_{(i,p)(k,p)}, \quad G=[g_{l,p}], g_{l,p}=\sum_{i=1}^{n_1} e_{(i,p)(i,l)},\\
 &\frac{1}{2}\sum_{i,j\in[n_1], p,q\in[n_2]} \big(\sum_{k=1}^{n_1}e_{(i,p)(k,q)}x_{k,j}-\sum_{l=1}^{n_2}e_{(i,p)(j,l)}y_{l,q}\big)\times\\
& \big(\sum_{k=1}^{n_1}e_{(j,q)(k,p)}x_{k,i}-\sum_{l=1}^{n_2}e_{(j,q)(i,l)}y_{l,p}\big)
\end{aligned}
\end{equation} 
For $n_1=n_2=n$ the complexity of finding the gradient and the Hessian of $f$ at $(0,0)$ is $O(n^6)$.
\end{enumerate}
\end{lemma}
\begin{proof} (a) Let $\lambda_i(Z),i\in[n]$ be the  eigenvalues $Z$.  Clearly, $\det(I+Z)=\prod_{i=1}^n(1+\lambda_i(Z))$.  
If $\|Z\|_2<1$ then $|\lambda_i(Z)|<1$ for $i\in[n]$.   Recall that for a complex $t, |t|<1$ one has the Taylor expansion $-\log(1+t)=\sum_{j=1}^{\infty}\frac{(-1)^{j}t^j}{j}$.  Hence
\begin{equation*}
\begin{aligned}
&-\log\det(I+Z)=-\sum_{i=1}^n\log(1+\lambda_i(Z))=\\
&-\sum_{i=1}^n \sum_{j=1}^\infty \frac{(-1)^j\lambda_i^j(Z)}{j}=\sum_{j=1}^\infty \frac{(-1)^j\tr (Z^j)}{j}.
\end{aligned}
\end{equation*}
This proves \eqref{logdetexp1}.  The proof of the rest of (a) is straightforward.

\noindent
(b)  Clearly,
\begin{equation*}
-\log\det (E+W)=-\log\det E E^{-1}(E+W)=-\log\det E -\log\det (I+E^{-1}W).
\end{equation*}
Observe that $\|Z\|_2\le \|E^{-1}\|_2\|W\|_2<1$.   This proves \eqref{expldEW}.
Let $E^{-1}=[\tilde e_{i,j}]\in \rH_{n}$.  Recall the complexity of finding $E^{-1}$ is $O(n^3)$ \cite{GV96}.  Assume that we computed the entries of $E^{-1}$.
Observe that $z_{i,j}=\sum_{k=1}^n \tilde e_{i,k}w_{k,j}$.   Use \eqref{logdetex} to deduce that the linear and quadratic terms in the Taylor expansion of $f(W)$ are 
\begin{equation*}
\begin{aligned}
&\sum_{i=1}^n\sum_{k=1}^n \tilde e_{i,k}w_{k,i},\\
& \sum_{i=j=1}^n\left(\sum_{k=1}^n  \tilde e_{i,k}w_{k,j}\right)\left(\sum_{l=1}^n  \tilde e_{j,l}w_{l,i}\right).
\end{aligned}
\end{equation*}
Thus,  to compute $\nabla(f)(0)$ we need $O(n^2)$ flops.  To compute $H(f)(0)$ we need $O(n^4)$ flops.  Taking in account the computation of $E^{-1}$, one needs $O(n^3)$  and $O(n^4)$ flops to compute to compute $\nabla(f)(0)$ and $H(f)(0)$ respectively.

\noindent
(c) The equalities in (c) are straightforward.  Assume that $n_1=n_2=n$.  As $C(A,B)\in \rH_{n^2,++}$ it follows that the computation of $C^{-1}(A,B)$ is $O(n^6)$.  Assume that we computed $C^{-1}(A,B)$.   Then, to compute each element of $F$ and $G$ we need $n$ flops Hence to compute $\nabla f(0)=(F,G)$ we need $O(n^3)$.  To compute $H(f)(0)$ we need to estimate the number of monomials we obtain in the quadratic terms of the Taylor expansion.  We have $n^4$ of products of two linear forms each containing $2n$ linear terms.  Hence we have $4n^6$ monomial in the quadratic expansion of $f(W)$.  That is, computing $H(f)(0)$ is $O(n^6)$.  Recalling that the computation of $C^{-1}(A,B)$ is $O(n^6)$ we deduce the last claim of (c). 
\end{proof}
\begin{theorem}\label{IPMqBP}
Let $\rho_i\in\rH_{n_i,++,1},i\in[2]$.  Consider the minimum \eqref{qotp}, which corresponds to the maximum dual problem \eqref{dqBOTP}.   Apply the IPM short step interior path algorithm with the barrier \eqref{hbqBT} to the $-\min$ problem as in \eqref{min=maxcBP}, where $r$ is given by
\eqref{defrqBP}.  Choose a  starting point  given by \eqref{X0Y0qBP}.
The CPA algorithm finds the value $\kappa(C,R)$ within precision $\delta>0$ in  
\begin{equation}\label{ipmqotmat1}
O\big(\sqrt{n_1n_2+1}\log\frac{r(n_1n_2+1)}{\delta}\big)
\end{equation}
iterations. 
For $n_1=n_2=n$ the arithmetic complexity of CPA is $O\big(n^7\log\frac{n^2r}{\delta}\big)$.
\end{theorem}
\begin{proof} Lemma \ref{maxqbest} shows that the maximum \eqref{dqBOTP} is achieved on $\rD_1$.  
Recall that the IPM method needs the number of iterations given in \eqref{Renthm}.
Part (f) of Lemma \ref{maxqbest} yields that $\textrm{sym}((\X_0,Y_0),\rD_1)\ge \sqrt{2}r$.
This shows that the number of iterations to find the value $\kappa(C,R)$ within precision $\delta>0$ is given by \eqref{ipmqotmat1}.
Assume that $n_1=n_2=n$.
Each step of the iteration of IPM needs the computation of the gradient and the Hessian of $\hat\beta$.   We need first to estimate the arithmetic complexity of computing $C(X,Y)^{-1}$ for $(X,Y)\in \hat\rD$.   As $C(X,Y)$ is an $n^2\times n^2$ it follows that the computation of  $C(X,Y)^{-1}$ is $O(n^6)$.   Part (c) of Lemma \ref{loddetder} yields that  computation of the gradient and the Hessian of $\hat\beta$ needs $O(n^6)$ flops.
The computation og the inverse of Hessian needs $O(n^6)$ flops.  Hence, each step of IPM needs $O(n^6)$ flops.  Therefore, the arithmetic complexity of the IPM method is  $O\big(n^7\log\frac{n^2r}{\delta}\big)$.
\end{proof}
\section{Quantum multipartite-partite optimal control and its relaxations}\label{sec:qmotp}
In this section we discuss briefly the extensions of our results in Section \ref{sec:qbotp}.  The proofs are omitted, and in some cases we point out modifications to our arguments that are given in Section \ref{sec:qbotp}. 

The quantum MOPTP is stated in \eqref{qotp}.  The analog of Lemma \ref{minubqbpotplem} is:
\begin{lemma}\label{minubqmpotplem}  Assume that $d\ge 3$.
Let $\rho_i\in\rH_{n_i,++},n_i>1,i\in[d]$  and assume that $\Gamma(R)$ be defined by \eqref{couprs}.
\begin{enumerate}[(a)]
\item 
Consider the minima $\kappa(C,R)$ and $\kappa_{\beta}(C,R,\varepsilon)$ that are given by \eqref{qotp} and \eqref{qbMOTP}, respectively. Then \eqref{bqtauin} holds.
\item The unique minimizing matrix $\rho^\star\in \Gamma_o(R)$ for \eqref{qbMOTP}, is of the form 
\begin{equation}\label{qbMPOTus}
\begin{aligned}
&\rho^\star=
\varepsilon \left(C-\sum_{j=1}^d(\otimes_{k=1}^{j-1} I_{n_k})\otimes U_j\otimes(\otimes_{k=j+1}^d I_{n_k})\right)^{-1}\in \Gamma_o(R),\\
&U_i\in\rH_{n_i}, i\in[d].
\end{aligned}
\end{equation}
Furthermore,  there exist a unique solution of the above form satisfying $\tr U_1=\cdots=\tr U_d$.
\end{enumerate}
\end{lemma}

In what follows we assume $n_1=\cdots=n_d=n\ge 2$ for simplicity of our exposition.
Let $\bn=(\underbrace{n,\ldots,n}_d)$
The following theorem is a generalization of Theorem \ref{dbqBPOTP}:
\begin{theorem}\label{dbqMPOTP} Let $(\rho_1,\ldots,\rho_d)\in (\times ^d \rH_{n})_{++,b}$.
For a given $C\in\rH_{\bn}$ define
\begin{equation}\label{defqdDhD}
\begin{aligned}
&C(X_1,\cdots,X_d)=C-\sum_{j=1}^d(\otimes_{k=1}^{j-1}I_{n})\otimes X_j\otimes(\otimes_{k=j+1}^d, I_{n}),\\
&(X_1,\ldots,X_d)\in\times^d\rH_n,\\
&\hat\rD=\{(X_1,\ldots,X_d)\in\times^d\rH_n, C(X_1,\ldots,X_d)\in \rH_{\bn,++}\},\\
&\rD=\{(X_1,\ldots,X_d)\in\hat\rD,\tr X_1=\cdots=\tr X_d, \}.
\end{aligned}
\end{equation} 
The barrier function for $\hat \rD$ is
\begin{equation}\label{barqMOTd}
\beta_Q(X_1,\ldots,X_d)=-\log\det C(X_1,\ldots,X_d),  \quad (X_1,\ldots,X_d)\in\hat\rD.
\end{equation}
The function 
\begin{equation}\label{mfbqBOTP}
\begin{aligned}
\varphi(X_1,\ldots,X_d,C)=\sum_{j=1}^d \tr \rho_jX_j-
\varepsilon\beta_Q(X_1,\ldots,X_d)+n^d\varepsilon(1-\log\varepsilon),
\end{aligned}
\end{equation}
is concave in $\hat\rD$ and strictly concave in $\rD$.  Furthermore,
\begin{equation}\label{maxqfMOTP}
\begin{aligned}
&\max_{(X_1,\ldots,X_d)\in\hat\rD}\varphi(X_1,\ldots,X_d,C)=\\
&\max_{(X_1,\ldots,X_d,C)\in \rD}\varphi(X_1,\ldots,X_d,C)=\kappa_{\beta}(C,R,\varepsilon),
\end{aligned}
\end{equation}
where $\beta$ and $\kappa_{\beta}(C,R,\varepsilon)$ are given \eqref{barpdm} and \eqref{qbMOTP} respectively.
The function $\varphi$ achieves its maximum
exactly on the hyperplane
\begin{equation}\label{defqLMP}
\rL=\{(U_1,\ldots,U_d)+(t_1I_{n_1}, \ldots,t_d I_{n_2}),(t_1,\ldots,t_d\in\R^d, \sum_{j=1}^dt_j=0\},
\end{equation}
which intersects $\rD$ at the unique point $(U_1,\ldots,U_d)$.  Furthermore, 
\begin{equation}\label{qbMPOTus}
\rho^\star=\varepsilon C(U_1,\ldots,U_d)^{-1}, (U_1,\ldots,U_d) \in \rL,
\end{equation}
 is the unique point in $\Gamma_o(R)$ at which $\kappa_{\beta}(C,R,\varepsilon)$ is achieved.
\end{theorem}
\subsection{The interior point method for finding a barrier approximation of $\kappa(C,R)$}\label{subsec:ipmqMP}
The following lemma is a generalization of Lemma \ref{maxqbest}
\begin{lemma}\label{maxqmbest}
Let $2\le n, 3\le d$ be integers, and $0\ne C\in\rH_{\bn}$. 
\begin{enumerate}[(a)]
\item Consider the minimum problem \eqref{qotp}.
Then
\begin{equation}\label{lbtau}
 \kappa(C,R)\ge \lambda_{\min}(C).
\end{equation}
\item Assume that $C(X_1,\ldots,X_d), \hat \rD,\rD$ are defined in \eqref{defqdDhD}.   
The dual maximum problem to \eqref{qotp} is given by 
\begin{equation}\label{dqMOTP}
\begin{aligned}
&\kappa(C,R)=\max_{(X_1,\cdots,X_d)\in\times^d\rH_n, C(X_1,\ldots,X_d)\succeq 0}\sum_{j=1}^d\tr\rho_jX_j=\\
&\sup_{(X_1,\cdots,X_d)\in\times^d\in\hat\rD} \sum_{j=1}^d\tr\rho_jX_j=
\sup_{(X_1,\cdots,X_d)\in\times^d\in\rD} \sum_{j=1}^d\tr\rho_jX_j.
\end{aligned}
\end{equation} 
\item  Let $\rho_i\in\rH_{n_i,++,1}, i\in[d]$.   Suppose that
\begin{equation}\label{X1Xdlbin}
\begin{aligned}
&(C(X_1,\ldots,X_d)\succeq 0)\wedge(\sum_{j=1}^d\tr\rho_jX_j\ge \lambda_{\min}(C))\wedge\\
&(\tr X_1=\cdots=\tr X_d).
\end{aligned}
\end{equation}
Then
\begin{equation}\label{X1Xdineq}
\begin{aligned}
\|(X_1,\ldots,X_d)\|_F=\sqrt{\sum_{j=1}^d\|X_j\|_F^2}<
\frac{2\|C\|_2\sqrt{n}}{\prod_{j=1}^d\lambda_{\min}(\rho_j)}.
\end{aligned}
\end{equation}
\item Let
\begin{equation}\label{X10Xd0qBP}
X_{j,0}=\frac{t}{n}I_{n}, j\in[d], \quad t=\frac{n(-1+\lambda_{\min}(C))}{d}.
\end{equation}
Then $(X_{1,0},\ldots,X_{d,0})\in\rD$. 
\item  Let 
\begin{equation}\label{defrqMP}
r^2=\frac{(9\|C\|_2^2+1)n}{(\prod_{j=1}^d\lambda_{\min}(\rho_j))^2}.
\end{equation}
The interior of the ball
$\rB((X_{1,0},\ldots,X_{d,0}),r)=\{\sum_{j=1}^d\|X_j-X_{j,0}\|_F^2\le  r^2\},$
denoted as $\rB_o((X_{1,0},\ldots,X_{d,0}),r)$, contains all the points $(X_1,,\ldots,X_d)$ that satisfies the conditoins  \eqref{X1Xdlbin}.
\item Let $\rD_1=\rD\cap   \rB_o((X_{1,0},\ldots,X_{d,0}),r)$.  Then for any boundary point $(X_1,\ldots,X_d)$ of $\rD_1$ the inequality $\|(X_{1,0},\ldots,X_{d,0})-(X_1,\ldots,X_d)\|_F\ge \frac{1}{\sqrt{d}}$ holds.
\item The function
\begin{equation}\label{hbqMTd}
\begin{aligned}
&\hat\beta(X,Y)=-\log\det C(X_1,\ldots,X_d)\\
&-\log\left(r^2-\sum_{j=1}^d\|X_j-X_{j,0}\|_F^2\right)+\log r^2
\end{aligned}
\end{equation}
is a barrier function for the convex domain $\rD_1$ with $\theta(\hat\beta)\le n^d+1$.
\end{enumerate}
\end{lemma}

The following proposition is a generalization of Proposition \ref{kapbepsest}.
\begin{proposition}\label{kapmepsest}  Let $2\le n, 3\le d\in\N$, and $0\ne C\in\rH_{\bn}$.  Assume that $\rho_i\in\rH_{n,++,1}, i\in[d]$ , and $\varepsilon>0$.  Then 
\begin{equation}\label{bmqtauin}
\begin{aligned}
&\kappa(C,R)<\kappa_{\beta}(C,R,\varepsilon)<\kappa(C,R) +n^d\varepsilon(1-\log\varepsilon)+\\
&n^d\varepsilon\log \left(\frac{2\|C\|_2}{\prod_{i=1}^d\lambda_{n}(\rho_i)}\right)
\end{aligned}
\end{equation}
\end{proposition}

The following theorem is a generalization of Theorem \ref{IPMqBP}
\begin{theorem}\label{IPMqMP}
Let $\rho_i\in\rH_{n,++,1},i\in[d]$.  Consider the minimum \eqref{qotp}, which corresponds to the maximum dual problem \eqref{dqMOTP}.   Apply the IPM short step interior path algorithm with the barrier \eqref{hbqMTd} to the $-\min$ problem as in \eqref{min=maxcBP}, where $r$ is given by
\eqref{defrqMP}.  Choose a  starting point  given by \eqref{X10Xd0qBP}.
The CPA algorithm finds the value $\kappa(C,R)$ within precision $\delta>0$ in  
$O\big(n^{d/2}\log\frac{rn^d}{\delta}\big)$ iterations. 
The arithmetic complexity of CPA is $O\big(n^{7d/2}\log\frac{n^dr}{\delta}\big)$.
\end{theorem}

The proof of this Theorem is similar to the proof Theorem \ref{IPMqBP}.  We discuss briefly the arithmetic complexity of the IPM algorithm.  Note that the matrix $C(X_1,\ldots,X_d)$ is $n^d\times n^d$.  Hence, its inverse is computed in $O(n^{3d})$ flops.  The computation of the gradient and the Hessian of $-\log (r^2-\sum_{i=1}^d \|X_i\|_F^2)+\log r^2$ needs $O(dn^2)$ and $O((dn)^2)$ flops respectively. It is left to discuss the computation 
of the gradient and Hessian $\beta$.

Assume that $C^{-1}(A_1,\ldots, A_d)=[e_{(i_1,\ldots,i_d)(j_1,\ldots,j_d)}]$.   
 Observe
\begin{equation}\label{linform}
\begin{aligned}
-C^{-1}(A_1,\ldots, A_d)(X_1\otimes I_{n^{d-1}})_{(i_1,\ldots,i_d)(j_1,\ldots,j_d)}=\\
\sum_{k_1=1}^n e_{(i_1,i_2,\ldots,i_d)(k_1,j_2,\ldots,j_d)}x_{k_1,j_1,1}
\end{aligned}
\end{equation}
Hence,  
\begin{equation*}
-\frac{\partial\log\det C(A_1+X_1,\ldots A_d+X_d)}{\partial x_{k_1,i_1}} 
=\sum_{j_2=\ldots=j_d=1}^n e_{(i_1,j_2,\ldots,j_d)(k_1,j_2,\ldots,j_d)}.
\end{equation*}
Thus to compute the gradient of $\beta$ we need $O(dn^{d+1})$ flops.  
To compute the number of flops for the Hessian of $\beta$ we use a similar equality to
\eqref{drasHesXY}.   It will contain $n^{2d}$ summands, where each summand is a product of two linear forms, where each linear form contains $dn$ terms corresponding to the entries of $X_1,\ldots,X_d$.   A linear term corresponding to $X_1$ is given in \eqref{linform}
 Hence the number of quadratic monomials is $n^{2d}(dn)^2=d^2n^{2d+2}$.
 Thus, the number of flops to compute $H(\beta)$ is $O(d^2n^{2d+2}$.   As $H(\beta)$ is a matrix of order $dn^2$, the computation of $H^{-1}(\beta)$ need $O((dn^2)^3)$ flops.
 Taking in account the inverse of   $C(X_1,\ldots,X_d)$ is  $O(n^{3d})$ flops we deduce that each step of the IPM method needs $O(n^{3d})$ flops.  As the IPM method needs
 $O\big(n^{d/2}\log\frac{rn^d}{\delta}\big)$ iterations., we deduce that
the arithmetic complexity of CPA is $O\big(n^{7d/2}\log\frac{n^dr}{\delta}\big)$.

\section*{Acknowledgment} 

\noindent
Part of this work was supported by IPAM in UCLA,  and conducted when the author was visiting  the Workshops I and III in Non-commutative Optimal Transport Program, Spring 2025.

\bibliographystyle{plain}

\appendix

\section{The interior point method}\label{sec:ipm}
We first recall some notations and definitions that we will use in this section.
\begin{equation}\label{defellsnrm}
\begin{aligned}
&\|\x\|_s=\bigl(\sum_{i=1}^n |x_i|^s\bigr)^{1/s}, \, s\in[1,\infty],\,
\x=(x_1,\ldots,x_n)^\top\in\R^n,\\
&\rB(\x,r)=\{\y\in\R^n, \|\y-\x\|\le r\} \textrm{ for } r\ge 0.
\end{aligned}
\end{equation}

Let $f\in\rC^3(\rB(\x,r))$ for $r>0$.    Denote 
\begin{equation*}
f_{,i_1\ldots i_d}(\x)=\frac{\partial ^d}{\partial x_{i_1}\ldots\partial x_{i_d}}f(\x), \quad i_1,\ldots,i_d\in[n], d\in[3].
\end{equation*}
Recall the Taylor expansion of $f$ at $\x$ of order $3$ for $\bu\in\R^n$ with a small norm:
\begin{equation*}
\begin{aligned}
&f(\x+\bu)\approx f(\x)+\nabla f(\x)^\top\bu+\frac{1}{2}\bu^\top \partial^2 f(\x)\bu+\frac{1}{6}\partial ^3 f(\x)\otimes \bu^{3\otimes},\\
&\nabla f(\x)=(f_{,1}(\x),\ldots,f_{,n}(\x))^\top,  \quad \partial^2 f(\x)=[f_{,ij}(\x)],  i,j\in[n],\\ 
&\partial^3 f(\x)=[f_{,ijk}(\x)], i,j,k\in[n], \, \partial^3 f(\x)\otimes\bu^{3\otimes}=\sum_{i,j,k\in[n]}f_{,ijk}(\x)u_i u_j u_k, 
\end{aligned}
\end{equation*}
where $\nabla f, \partial^2 f, \partial ^3 f$ are called
the gradient, the Hessian, and the 3-mode symmetric partial derivative tensor of $f$. 
A set $\rD\subset \R^n$ is called a domain if $\rD$ is an open connected set.
\begin{definition}\label{defconcconst}  
Assume that $f: \rD\to\R$ is a convex function 
 in a convex domain $\rD\subset \R^n$,  and $f\in\rC^3(\rD)$.   The function $f$ is called   $a (>0)$-self-concordant, or simply self-concordant, if the following inequality hold
\begin{equation}\label{defconcconst1} 
|\langle \partial^3f(\x),\otimes^3\bu\rangle|\le 2a^{-1/2} (\bu^\top \partial^2f(\bx)\bu)^{3/2}, 
\textrm{ for all }\x\in \rD,\bu\in\R^n.
\end{equation}
The function $f$ is called a standard self-concordant if $a=1$,  and a strongly $a$-self-concordant if $f(\x_m)\to\infty$ if the sequence $\{\x_m\}$ converges to the boundary of $\rD$.

The complexity value $\theta(f)\in[0,\infty]$ of an a-self-concordant function $f$ in $\rD$, called a self-concordant parameter in \cite[Definition 2.3.1]{NN94}, is
\begin{equation}\label{defsconc0}
\begin{aligned}
\theta(f)=sup_{\x\in\rD}
\inf\{\lambda^2\in[0,\infty], |\nabla f(\x)^\top \bu|^2\\
\le \lambda^2 a\big(\bu^\top\partial^2 f(\x)\bu\big),\forall \bu\in\R^n\}.
\end{aligned}
\end{equation}

 A strongly self-concordant function with a finite $\theta(f)$ is called a barrier (function).
\end{definition}

The following equality is well known  \cite[top of page 16]{NN94}.
\begin{proposition}\label{charthetf}  Let $\rD\subset\R^n$ be a convex domain, and assume that $f$ is an $a$-self-concordant function in $\rD$.    Suppose furthermore that the Hessian $\partial^2(f)$ is positive definite for each $\x\in\rD$.   Then
\begin{equation}\label{defconcconst1}
 \theta(f)=a^{-1}\sup_{\x\in \rD}\nabla f(\x)^\top (\partial^2f)^{-1}(\x)\nabla f(\x).
\end{equation}
\end{proposition}

The above result yields that if  $f_i$ is $a$-self-concordant barrier in $\rD_i\subset \R^n$ for $i\in[N]$, such that $\rD=\cap_{i=1}^N \rD_i\ne\emptyset$ then $f=\sum_{i=1}^N f_i$ is an $a$-self-concordant barrier in $\rD$, with $\theta(f)\le \sum_{i=1}^N \theta(f_i)$.

Denote by $\rS_n\supset\rS_{n,+}\supset \rS_{n,++}$ the space of $n\times n$ symmetric, the cone of positive semidefinite and the open set of positive definite matrices respectively.   

A first example of a strongly standard self-concordant function for the interior of the cone $\R_+^n$,  denoted as $\R_{++}^n$, where $\R_{++}=(0,\infty)$, is 
\begin{equation}\label{logxbar}
\begin{aligned}
&\sigma(\x)=-\sum_{i=1}^n \log x_i,\\
&\theta(\sigma)=n.
\end{aligned}
\end{equation}
A second example of a strongly standard self-concordant function for $\rS_{,++}$ is 
$-\log\det X$, where
\begin{equation}\label{-logdetX}
\begin{aligned}
\theta(-\log\det X)=n, \quad X\in\rS_{n,++}.
\end{aligned}
\end{equation}

 Assume that $\rD$ is a bounded convex domain,  $\x\in\rD$  and $\rL$ is a line  through $\x$. 
 Denote by $d_{\max}(\x,\rL)\ge d_{min}(\x,\rL)$ the two distances from $\x$ to the end points of $\rL\cap\partial\rD$.   Then $\textrm{sym}(\x,\rD)$ is the infimum of $\frac{d_{\min}(\x,L)}{ d_{\max}(\x,L)}$ for all lines $\rL$ through $\x$.   Observe that if $\rB(\x,r)\subset \textrm{Closure}(\rD)\subset \rB(\x,R)$ then $\textrm{sym}(\x,\rD)\ge r/R$.

Recall that Renegar \cite{Ren01} deals only with strongly standard self-concordant functions.  
The complexity value $\theta(f)$, coined in \cite{Ren01}, is called the parameter of barrier $f$ in \cite{NN94}, and is considered only for self-concordant barrier  in \cite[\textsection 2.3.1] {NN94}.   

We now recall the  complexity result to approximate the infimum of a linear functional on a bounded convex domain with whose boundary is given by a barrier function $\beta$.
We normalize $\beta$ by assuming that it is strongly self-concordant.
A simple implementation of the Newton's method  is  a "short-step" ipm's that  follows the central path \cite[\textsection 2.4.2]{Ren01}.
The number of iterations to approximate the minimum of a linear functional  within $\delta$ precision starting with an intial point $\x'$ is \cite[Theorem 2.4.1]{Ren01}:
\begin{equation}\label{Renthm}
O\big(\sqrt{\theta(\beta)}\log\big(\frac{\theta(\beta)}{\delta \textrm{sym}(\x',\rD)}\big)\big).
\end{equation}
\section{Numerical complexity of algorithms in terms of flops}\label{sec:flop}
In this paper we estimate the complexity of algorithms  in terms of 
of floating point operations or flops.   We use the standard estimations for the number of flops for vector and matrix operations as in \cite{GV96}.  In particular, the determinant and inverse of $n\times n$ matrix, and the  solution of $n$- equations with an invertible matrix needs $O(n^3)$ flops.

\end{document}